\documentclass[11pt]{amsart}
\usepackage{xcolor,amssymb,verbatim}
\usepackage{afterpage}
\usepackage[all,cmtip]{xy}
\usepackage[margin=1in]{geometry}
\usepackage[colorlinks=true, pdfstartview=FitV, linkcolor=blue, citecolor=blue, urlcolor=blue]{hyperref}

\newcommand{\arxiv}[1]{\href{http://arxiv.org/abs/#1}{\texttt{arXiv:#1}}}

\def \Bzero{{0}}
\renewcommand{\l}[1] {
l^{J}_{#1}
}
\def \A{\mathcal{A}}
\def \reals{\mathbb{R}}
\newcommand{\g}{\mathfrak{g}}

\newcommand{\clfw}{\omega} % classical fundamental weights
\newcommand{\inner}[2]{\left\langle #1, #2 \right\rangle}
\newcommand{\iso}{\cong}

\newcommand{\sgn}{\operatorname{sgn}}

\newcommand{\virtual}[1]{\widehat{#1}}

\DeclareMathOperator{\wt}{wt} % weight
\DeclareMathOperator{\im}{Im} % Image
\DeclareMathOperator{\cl}{cl} % classical
\newcommand{\ZZ}{\mathbb{Z}}

\newcommand{\bon}{\overline{1}}

\definecolor{darkred}{rgb}{0.7,0,0} % darkred color
\newcommand{\defn}[1]{{\color{darkred}\emph{#1}}} % emphasis of a definition

\theoremstyle{plain}
\newtheorem{thm}{Theorem}[section]
\newtheorem{lemma}[thm]{Lemma}
\newtheorem{conj}[thm]{Conjecture}
\newtheorem{prop}[thm]{Proposition}
\newtheorem{cor}[thm]{Corollary}
\theoremstyle{definition}

\newtheorem{ex}[thm]{Example}
\newtheorem{remark}[thm]{Remark}
\newtheorem{problem}[thm]{Problem}
\newtheorem{remarks}[thm]{Remarks}
\numberwithin{equation}{section}
\usepackage[colorinlistoftodos]{todonotes}

\begin{document}
\title[On higher level KR crystals]{On higher level Kirillov--Reshetikhin crystals,\\ Demazure crystals, and related uniform models}

\author[C.~Lenart]{Cristian Lenart}
\address[C. Lenart]{Department of Mathematics, State University of New York at Albany, Albany, NY 12222, USA}
\email{clenart@albany.edu}
\urladdr{https://www.albany.edu/~lenart/}

\author[T.~Scrimshaw]{Travis Scrimshaw}
\address[T. Scrimshaw]{School of Mathematics and Physics, The University of Queensland, St.\ Lucia, QLD 4072, Australia}
\email{tcscrims@gmail.com}
\urladdr{https://people.smp.uq.edu.au/TravisScrimshaw/}


\begin{abstract}
We show that a tensor product of nonexceptional type Kirillov--Reshetikhin (KR) crystals is isomorphic to a direct sum of Demazure crystals; we do this  in the mixed level case and without the perfectness assumption, thus generalizing a result of Naoi. We use this result to show that, given two tensor products of such KR crystals with the same maximal weight, after removing certain $0$-arrows, the two connected components containing the minimal/maximal elements are isomorphic. Based on the latter fact, we reduce a tensor product of higher level perfect KR crystals to one of single-column KR crystals, which allows us to use the uniform models available in the literature in the latter case. We also use our results to give a combinatorial interpretation of the $Q$-system relations. Our results are conjectured to extend to the exceptional types.
\end{abstract}

\maketitle

%=====================================================================

\section{Introduction}
\label{sec:introduction}

\emph{Kirillov--Reshetikhin} (KR) \emph{modules} are a family of finite-dimensional representations of an affine quantum group without derivation $U_q'(\g)$ that are characterized by their Drinfel'd polynomials. They have been the subject of intense study, with numerous applications and properties, some of which are still conjectural to various extents. For example, see~\cite{CP95,CP98,FL06,FL07,FOS09,FOS10,Hernandez10,KKMMNN92,KNS11,LNSSS14, LNSSS14II,LNSSS15,LS17,Naoi13,OS08,OSS17,ST12} and the references therein. One of the most important conjectural properties~\cite{HKOTY99, HKOTT02} is that KR modules admit \emph{crystal bases} in the sense of Kashiwara~\cite{K90,K91,K94}. These crystals are called \emph{Kirillov--Reshetikhin} (KR) \emph{crystals} and are denoted $B^{r,s}$, where $r$ is an index of the classical Dynkin diagram of $\g$ and $s \in \ZZ_{>0}$. KR crystals have been shown to exist in all nonexceptional types by Okado and Schilling~\cite{OS08}, in types $G_2^{(1)}$ and $D_4^{(3)}$ by Naoi~\cite{Naoi17}, for certain $r$ in exceptional types~\cite{BS19II,JS10,Naoi19}, and for $r=1$ in all types by Kashiwara~\cite{K02}.

KR crystals and their tensor products are known to be connected with \emph{Demazure crystals} of affine highest weight representations. A precise description is known for a tensor product of KR crystals in nonexceptional types such that they are all \emph{perfect} of the same level~\cite{FSS07,KKMMNN91,ST12}; namely, this tensor product is isomorphic, up to certain 0-arrows, to a specified Demazure crystal. When the (nonexceptional type) KR crystals in the tensor product are perfect of mixed levels, Naoi showed that one obtains a direct sum of Demazure crystals~\cite{Naoi13}. In addition, this relationship was given for $B^{r,1}$ in all types~\cite{K05} and can be extended to tensor products by the techniques of~\cite{Naoi13} (see also~\cite{Naoi12} for an alternative proof). Further connections of KR crystals, viewed as classical crystals, were described in~\cite{FL06,FL07}.

One important unsolved problem involving KR crystals $B^{r,s}$ (and their tensor products) is constructing a uniform model in all types. For $B^{r,1}$ and their tensor products, such a model, based on \emph{projected level-zero Lakshmibai--Seshadri} (LS) \emph{paths}, was given by Naito and Sagaki~\cite{NS08II,NS08}. In untwisted types, an explicit description of these piecewise-linear paths was given as \emph{quantum Lakshmibai--Seshadri} (LS) \emph{paths} in~\cite{LNSSS14II}; the alternative \emph{quantum alcove model} was given in the same paper (see also~\cite{LL15,LL18}), while the quantum LS paths for type $A_{2n}^{(2)}$ were developed in~\cite{Nom16}. A partially uniform model for $B^{r,1}$ using \emph{Nakajima monomials} was given by Hernandez and Nakajima~\cite{HN06}. A uniform model for the classical crystal structure of $B^{r,s}$ was given in terms of \emph{rigged configurations} in~\cite{Kleber98,OSS03II,S06,SchillingS15}. However, the affine crystal structure on rigged configurations has currently only been explicitly constructed for type $A_n^{(1)}$~\cite{SW10}, $D_n^{(1)}$~\cite{OSS13}, and $A_{2n-1}^{(2)}$, $B_n^{(1)}$~\cite{SchillingS15}. Also in the nonexceptional case, a type-specific construction, based on virtual crystals and tableaux, is found in~\cite{FOS09}.

One goal of this paper is to construct the perfect KR crystals $B^{r,s}$ of nonexceptional type, as well as their tensor products of a fixed level, up to certain $0$-arrows called \emph{non-level $\ell$ Demazure arrows}, respectively \emph{non-level $\ell$ dual Demazure arrows}. This construction is done by identifying them with specific subcrystals of certain tensor products of single-column KR crystals $B^{r,1}$. Then, for the latter, we can use the uniform models mentioned above, \textit{i.e.}, quantum LS paths and the quantum alcove model. We focus on the latter, as it is purely combinatorial and easier to use. Furthermore, we are currently working on a very explicit combinatorial description of the mentioned subcrystals.

The paper also achieves several other goals as follows. First, we derive as our main tool a generalization of Naoi's result mentioned above~\cite{Naoi13} to the nonperfect case. Secondly, our reduction theorem used to construct $B^{r,s}$ in terms of $B^{r,1}$ is proved in much larger generality, as an identification between two tensor products of KR crystals of mixed levels (again, of nonexceptional type, and possibly nonperfect). Thirdly, another special case of this relationship is shown to realize combinatorially a part of the $Q$-system relations, which are satisfied by the classical characters of KR crystals~\cite{Hernandez10}; we are led to a conjecture about a combinatorial realization of the entire $Q$-system relations. 

We conjecture that our results extend to the exceptional types. In particular, this would immediately lead to a uniform model for all (level $\ell$ dual Demazure portions of) tensor products of perfect KR crystals with a fixed level. Other problems are stated as well.

Let us describe our results in more detail. For a tensor product of KR crystals $B$, there exists a unique (classical) weight $\lambda$, called the maximal weight, and a unique element of weight $w_0(\lambda)$, called the minimal element; here $w_0$ is the longest element of the corresponding finite Weyl group.
We say that a $0$-arrow is a non-level $\ell$ Demazure arrow if it is one of the first $\ell$ $0$-arrows in its $0$-string. Given two tensor products of KR crystals $B$ and $B'$ with the same maximal weight $\lambda$ and of level bounded by $\ell$, our main theorem states that, after removing all non-level $\ell$ Demazure arrows, the connected components containing the corresponding minimal elements are isomorphic. A contragredient dual version of this result also holds.

The main tool in proving our construction is showing for a tensor product of KR crystals $B$ that $B \otimes u_{\ell \Lambda_0}$ (this tensor product is equivalent to removing all the non-level $\ell$ Demazure arrows) is isomorphic to a direct sum of Demazure crystals; this is the mentioned generalization of Naoi's result. As a consequence, we show that all tensor products of KR crystals are isomorphic to some direct sum of Demazure crystals. We note that our results do not imply that $B^{r,s}$, when perfect of level $\ell$, is isomorphic to a single Demazure crystal; however, this does follow as a consequence when the two crystals have the same classical characters, like in the cases discussed in~\cite{FL06}.

This paper is organized as follows.
In Section~\ref{sec:background}, we provide the necessary background.
In Section~\ref{sec:results}, we prove our main results.
In Section~\ref{models} we describe the reduction to single-column KR crystals and explain the way in which the quantum alcove model applies to higher level KR crystals.
In Section~\ref{sec:Q_system}, we refer to the $Q$-system relations and the mentioned conjecture involving them.

\subsection*{Acknowledgments}

C.L.\ was partially supported by the NSF grant DMS--1362627 and the Simons Foundation grant \#584738. T.S.\ was partially supported by the Australian Research Council grant DP170102648. C.L.\ gratefully acknowledges the hospitality of Institut des Hautes \'Etudes Scientifiques (France), where part of this work was carried out. T.S.\ thanks SUNY Albany for its hospitality during his visit in January 2018. The authors thank the referee for useful comments. This work benefited from computations using \textsc{SageMath}~\cite{sage,combinat}.

%=====================================================================
\section{Background}
\label{sec:background}

Let $\g$ be an affine Kac--Moody Lie algebra with index set $I$, Cartan matrix $(A_{ij})_{i,j \in I}$, simple roots $(\alpha_i)_{i \in I}$, fundamental weights $(\Lambda_i)_{i \in I}$, weight lattice $P$, simple coroots $(\alpha_i^{\vee})_{i \in I}$, and canonical pairing $\langle\ ,\ \rangle \colon P^{\vee} \times P \to \ZZ$ given by $\inner{\alpha_i^{\vee}}{\alpha_j} = A_{ij}$.
We write $i \sim j$ if $A_{ij} \neq 0$ and $i \neq j$.
Let $U_q(\g)$ denote the corresponding (Drinfel'd--Jimbo) quantum group.
Define $c_i^{\vee} := \max(a_i^{\vee}/a_i, 1)$, where $a_i$ and $a_i^{\vee}$ are the Kac and dual Kac labels, respectively~\cite[Table Aff1-3]{kac90}.
Let $P^+$ and $P^-$ denote the positive and negative weight lattices, respectively. We denote by $P_\ell^+$ the dominant weights of level $\ell$. 
Let $Q$ be the root lattice, with $Q^+$ and $Q^-$ being the positive and negative root lattices, respectively.
Let $W$ be the Weyl group corresponding to $\g$. The (strong) Bruhat order on $W$ has covers $w\lessdot ws_\alpha$ with $\ell(ws_\alpha)=\ell(w)+1$, where $\ell(\,\cdot\,)$ denotes the length function. 

The extended affine Weyl group is $\widetilde{W} := W \rtimes \Pi \iso W_0 \ltimes P_0$, where $\Pi$ is the set of length $0$ elements (in $\widetilde{W}$) and corresponds to automorphisms of the Dynkin diagram of $\g$. Let $t_{\mu} \in \widetilde{W}$ be the translation by $\mu \in P_0$. See, \textit{e.g.},~\cite{Bourbaki02,Carter05,kac90} for more information on the extended affine Weyl group.

Let $U_q'(\g) := U_q([\g, \g])$. Note that the corresponding weight lattice is $P' := P / \ZZ\delta$, where $\delta = \sum_{i \in I} c_i \alpha_i$ is the null root; in particular, the simple roots in $P'$ are linearly dependent.
We will sometimes abuse notation and write $P$ instead of $P'$ when there is no danger of confusion.

Let $\g_0$ denote the canonical simple Lie algebra given by the index set $I_0 = I \setminus \{0\}$, and $U_q(\g_0)$ the corresponding quantum group. Let $\cl \colon P \to P_0$ denote the natural classical projection onto the weight lattice $P_0$ of $\g_0$.
Let $\clfw_i := \cl(\Lambda_i)$ be the classical projection of the fundamental weight $\Lambda_i$.
Let $Q_0$ and $W_0$ be the root lattice and Weyl group of $\g_0$, respectively. As usual, we denote by $w_0$ the longest element of $W_0$. 

The \defn{quantum Bruhat graph}~\cite{FW04} is the directed graph on $W_0$ with edges labeled by positive roots of $\g_0$
\begin{equation}
	w \xrightarrow[\hspace{25pt}]{\alpha}
	ws_{\alpha} \quad 
	\text{ for  } 
	w \lessdot ws_{\alpha} \,\mbox{ or }\, \ell(ws_{\alpha}) =  \ell(w) - 2 \langle\rho,\alpha^\vee\rangle   + 1\,;
	\label{eqn:qbruhat_edge}
\end{equation}
here $\rho$ denotes, as usual, half the sum of the positive roots of $\g_0$.

%%%%%%%%%%
\subsection{Crystals}\label{sec:crystals}

An \defn{abstract $U_q(\g)$-crystal} is a set $B$ endowed with \defn{crystal operators} $e_i, f_i \colon B \to B \sqcup \{0\}$, for $i \in I$, and \defn{weight function} $\wt \colon B \to P$ that satisfy the following conditions:
\begin{itemize}
\item[(1)] $\varphi_i(b) = \varepsilon_i(b) + \inner{\alpha_i^{\vee}}{\wt(b)}$, for all $b \in B$ and $i \in I$,
\item[(2)] $f_i b = b'$ if and only if $b = e_i b'$, for $b, b' \in B$ and $i \in I$,
\item[(3)] $\wt(f_i b) = \wt(b) - \alpha_i$ if $f_i b \neq 0$;
\end{itemize}
where the statistics $\varepsilon_i, \varphi_i \colon  B \to \ZZ_{\geq 0}$ are defined by
\[
\varepsilon_i(b) := \max \{ k \mid e_i^k b \neq 0 \}\,,
\qquad \qquad \varphi_i(b) := \max \{ k \mid f_i^k b \neq 0 \}\,.
\]

\begin{remark}
The definition of an abstract crystal given in this paper is sometimes called a \defn{regular} or \defn{seminormal} abstract crystal in the literature.
\end{remark}

Let $e_i^{\max} b := e_i^{\varepsilon_i(b)} b$ and $f_i^{\max} b := f_i^{\varphi_i(b)} b$.
From the axioms, we identify $B$ with an $I$-edge colored weighted directed graph,
where there is an $i$-colored edge $b \to b'$ in the graph if and only if $f_i b = b'$.
Thus an entire $i$-string through an element $b \in B$ is given diagrammatically by
\[
e_i^{\max}b \xrightarrow[\hspace{15pt}]{i}
\cdots \xrightarrow[\hspace{15pt}]{i}
e_i^2 b \xrightarrow[\hspace{15pt}]{i}
e_i b \xrightarrow[\hspace{15pt}]{i}
b \xrightarrow[\hspace{15pt}]{i}
f_i b \xrightarrow[\hspace{15pt}]{i}
f_i^2 b \xrightarrow[\hspace{15pt}]{i}
\cdots \xrightarrow[\hspace{15pt}]{i}
f_i^{\max} b.
\]
An element $b \in B$ is \defn{highest (resp.\ lowest) weight} if $e_i b = 0$ (resp.~$f_i b = 0$) for all $i \in I$.
We say that $b \in B$ is \defn{classically highest (resp.\ lowest) weight} if $e_i b = 0$ (resp.~$f_i b = 0$) for all $i \in I_0$.

We define the \defn{tensor product} of abstract $U_q(\g)$-crystals $B_1$ and $B_2$ as the crystal $B_2 \otimes B_1$ that is the Cartesian product $B_2 \times B_1$ with the following crystal structure:
\begin{align*}
e_i(b_2 \otimes b_1) & := \begin{cases}
e_i b_2 \otimes b_1 & \text{if } \varepsilon_i(b_2) > \varphi_i(b_1)\,, \\
b_2 \otimes e_i b_1 & \text{if } \varepsilon_i(b_2) \leq \varphi_i(b_1)\,,
\end{cases}
\\ f_i(b_2 \otimes b_1) & := \begin{cases}
f_i b_2 \otimes b_1 & \text{if } \varepsilon_i(b_2) \geq \varphi_i(b_1)\,, \\
b_2 \otimes f_i b_1 & \text{if } \varepsilon_i(b_2) < \varphi_i(b_1)\,,
\end{cases}
\\ \varepsilon_i(b_2 \otimes b_1) & := \max(\varepsilon_i(b_1), \varepsilon_i(b_2) - \inner{\alpha_i^{\vee}}{\wt(b_1)})\,,
\\ \varphi_i(b_2 \otimes b_1) & := \max(\varphi_i(b_2), \varphi_i(b_1) + \inner{\alpha_i^{\vee}}{\wt(b_2)})\,,
\\ \wt(b_2 \otimes b_1) & := \wt(b_2) + \wt(b_1)\,.
\end{align*}

\begin{remark}
Our tensor product convention follows~\cite{BS17}, which is opposite to that of Kashiwara~\cite{K91}.
\end{remark}

For abstract $U_q(\g)$-crystals $B_1, \dotsc, B_L$, the action of the crystal operators on the tensor product $B := B_L \otimes \cdots \otimes B_2 \otimes B_1$ can be computed by the \defn{signature rule}.
Let $b := b_L \otimes \cdots \otimes b_2 \otimes b_1 \in B$, and for $i \in I$, we write
\[
\underbrace{-\cdots-}_{\varphi_i(b_L)}\ 
\underbrace{+\cdots+}_{\varepsilon_i(b_L)}\ 
\cdots\ 
\underbrace{-\cdots-}_{\varphi_i(b_1)}\ 
\underbrace{+\cdots+}_{\varepsilon_i(b_1)}\,.
\]
Then by successively deleting consecutive $+-$-pairs (in that order) in the above sequence, we obtain a sequence
\[
\sgn_i(b) :=
\underbrace{-\cdots-}_{\varphi_i(b)}\
\underbrace{+\cdots+}_{\varepsilon_i(b)}\,,
\]
called the \defn{reduced signature}.
Suppose $1 \leq j_-\leq j_+ \leq L$ are such that $b_{j_-}$ contributes the rightmost $-$ in $\sgn_i(b)$ and $b_{j_+}$ contributes the leftmost $+$ in $\sgn_i(b)$. 
Then, we have
\begin{align*}
e_i b &:= b_L \otimes \cdots \otimes b_{j_++1} \otimes e_ib_{j_+} \otimes b_{j_+-1} \otimes \cdots \otimes b_1\,, \\
f_i b &:= b_L \otimes \cdots \otimes b_{j_-+1} \otimes f_ib_{j_-} \otimes b_{j_--1} \otimes \cdots \otimes b_1\,.
\end{align*}

Let $B_1$ and $B_2$ be two abstract $U_q(\g)$-crystals.
A \defn{crystal morphism} $\psi \colon B_1 \to B_2$ is a map $B_1 \sqcup \{0\} \to B_2 \sqcup \{0\}$ with $\psi(0) = 0$, such that the following properties hold for all $b \in B_1$ and $i \in I$:
\begin{itemize}
\item[(1)] if $\psi(b) \in B_2$, then $\wt\bigl(\psi(b)\bigr) = \wt(b)$, $\varepsilon_i\bigl(\psi(b)\bigr) = \varepsilon_i(b)$, and $\varphi_i\bigl(\psi(b)\bigr) = \varphi_i(b)\,;$
\item[(2)] we have $\psi(e_i b) = e_i \psi(b)$ if $\psi(e_i b) \neq 0$ and $e_i \psi(b) \neq 0\,;$
\item[(3)] we have $\psi(f_i b) = f_i \psi(b)$ if $\psi(f_i b) \neq 0$ and $f_i \psi(b) \neq 0\,.$
\end{itemize}
An \defn{embedding} (resp.~\defn{isomorphism}) is a crystal morphism such that the induced map $B_1 \sqcup \{0\} \to B_2 \sqcup \{0\}$ is an embedding (resp.~bijection).
A crystal morphism is \defn{strict} if it commutes with all crystal operators.
Note that for a strict crystal embedding $\psi \colon B \to B'$ and connected components $C \subseteq B$ and $C' \subseteq B'$ such that $\psi(c) \in C'$ for any $c \in C$, the restriction $\widehat{\psi} \colon C \to C'$ is an isomorphism, i.e., connected components go to connected components under $\psi$.

A \defn{similarity map}~\cite{K96} is an embedding of crystals $\sigma=\sigma_m \colon B \to \virtual{B}$, with $m\in{\mathbb Z}_{>0}$, that satisfies
\begin{equation}\label{def:sim}
e_i \mapsto e_i^{m}\,,
\quad
f_i \mapsto f_i^{m}\,,\quad
 \varepsilon_i\bigl( \sigma(b) \bigr)=m \varepsilon_i(b)\,,
\quad
 \varphi_i\bigl( \sigma(b) \bigr)=m \varphi_i(b)\,,\quad
\wt(\sigma(b)) = m \wt(b)\,.
\end{equation}

An abstract crystal $B$ is a \defn{$U_q(\g)$-crystal} if $B$ is the crystal basis of some $U_q(\g)$-module.
Kashiwara~\cite{K91} has shown that the irreducible highest (resp. lowest) weight module $V(\lambda)$, for $\lambda \in P^+$ (resp.~$\lambda \in P^-$), admits a crystal basis, denoted $B(\lambda)$; this has a unique highest (resp.\ lowest) weight element $u_{\lambda}$ such that $\wt(u_{\lambda}) = \lambda$. The elements $u_{w\lambda} := w u_{\lambda}$, for $w\in W$, are called \defn{extremal}; here we used the $W$-action on the crystal, which was defined by Kashiwara~\cite{K94} as follows:
\[
s_i b := \begin{cases}
f_i^{\langle \alpha_i^{\vee}, \wt(b) \rangle} b & \text{if } \langle \alpha_i^{\vee}, \wt(b) \rangle > 0\,, \\ 
e_i^{-\langle \alpha_i^{\vee}, \wt(b) \rangle} b & \text{if } \langle \alpha_i^{\vee}, \wt(b) \rangle \leq 0\,.
\end{cases}
\]

%%%%%%%%%%
\subsection{Demazure crystals}

Let $\lambda \in P^+$. A Demazure module is a $U_q^+(\g)$-module generated by an extremal weight vector of weight $w \lambda \in V(\lambda)$. Kashiwara showed that the Demazure module has a crystal basis that is compatible with the crystal basis $B(\lambda)$ of the corresponding highest weight module $V(\lambda)$.

Hence, we can construct the crystal of a Demazure module as a subcrystal of $B(\lambda)$.
Fix a reduced expression $w = s_{i_1} s_{i_2} \cdots s_{i_{\ell}}$. A \defn{Demazure crystal} of the highest weight crystal $B(\lambda)$ is the full subcrystal given by
\begin{equation}\label{def:dem}
B_w(\lambda) := \{ b \in B(\lambda) \mid e_{i_1}^{\max} e_{i_2}^{\max} \cdots e_{i_{\ell}}^{\max} b = u_{\lambda} \}\,.
\end{equation}

\begin{thm}[Combinatorial excellent filtration~{\cite{Joseph03,LLM02}}]
\label{thm:excellent_filtration}
For all $\lambda, \mu \in P^+$, the crystal $B_w(\mu) \otimes u_{\lambda}$ is a direct sum of Demazure crystals.
\end{thm}

We also require the following fact, which follows from the definition of the Bruhat order on the Weyl group $W$.

\begin{prop}
\label{prop:bruhat_Demazure}
Let $\lambda \in P^+$.
We have $v \leq w$ if and only if $B_v(\lambda) \subseteq B_w(\lambda)$.
\end{prop}

%%%%%%%%%%
\subsection{Kirillov--Reshetikhin crystals}
\label{sec:kr}

Let $B^{r,s}$ denote the \defn{Kirillov--Reshetikhin (KR) crystal}, where $r \in I_0$ and $s \in \ZZ_{>0}$. We refer to Section~\ref{sec:introduction} for a review of the cases when the existence of KR crystals was proved, as well as of the related combinatorial models. 
%The KR crystal $B^{r,1}$ in all affine types was constructed uniformly using projected level-zero LS paths by the work of Naito and Sagaki~\cite{NS08II, NS08}.
%A combinatorial model for $B^{r,s}$ in all nonexceptional types was given in~\cite{FOS09}.

KR crystals have a number of conjectural properties. A KR crystal $B^{r,s}$ is conjectured~\cite{KKMMNN92} to be \defn{perfect}\footnote{The property of being perfect is a technical condition related to KR crystals, which is used to construct the Kyoto path model~\cite{KKMMNN92}; see also~\cite{BS17}.} of level $s / c_r$ if and only if $s / c_r \in \ZZ$. This has been shown for all nonexceptional types in~\cite{FOS10}, and in some special cases for other types~\cite{KMOY07,Yamane98}.
KR crystals are known to be well-behaved under similarity maps in nonexceptional types, as stated below.

\begin{thm}[{\cite{Okado13}}]
\label{thm:KR_similarity}
Let $\g$ be of nonexceptional affine type.
There exists a (unique) similarity map $\sigma_m \colon B^{r,s} \to B^{r,ms}$.
\end{thm}

There exists a unique classical component $B(s\clfw_r) \subseteq B^{r,s}$, and for any other classical component $B(\lambda) \subseteq B^{r,s}$, we have $s\clfw_r - \lambda \in Q_0^+$.
We remark that $B^{r,s} \iso B(s\clfw_r)$ as $U_q(\g_0)$-crystals whenever $r$ is the image of $0$ for some Dynkin diagram automorphism.
Let $u_{\max}(B^{r,s}) := u_{r\clfw_s} \in B(s\clfw_r) \subseteq B^{r,s}$ denote the \defn{maximal element}. For general $B := \bigotimes_{j=1}^N B^{r_j, s_j}$, the maximal element is defined as $u_{\max}(B) := u_{\max}(B^{r_1,s_1}) \otimes \cdots \otimes u_{\max}(B^{r_N,s_N})$, and note that it is the unique element of classical weight $\sum_{j=1}^N s_j \clfw_{r_j}$. 
Similarly, let $u_{\min}(B)$ denote the \defn{minimal element} of $B$, which is the unique element of classical weight $w_0 \!\left(\sum_{j=1}^N s_j \clfw_{r_j}\right)$.

The key property we need is a relationship between tensor products of KR crystals and Demazure crystals of an affine highest weight crystal.

\begin{thm}[\cite{FSS07,KKMMNN91,ST12}]
\label{thm:demazure_embedding}
Let $\g$ be of nonexceptional type.
Let $B := \bigotimes_{j=1}^N B^{r_j,s_j}$ such that there exists $\ell\in\ZZ$ with $s_j / c_{r_j} = \ell$ for all $j$.
Let $\clfw := -(c_{r_1} \clfw_{r_1^*} + \cdots + c_{r_N} \clfw_{r_N^*})$, where $\clfw_{r^*} = -w_0(\clfw_r)$.
Then, there exists a crystal isomorphism
\[
\psi \colon B(\ell \Lambda_{\tau(0)}) \to B \otimes B(\ell \Lambda_0)
\]
given by  $u_{\ell \Lambda_{\tau(0)}} \mapsto u_{g} \otimes u_{\ell \Lambda_0}$, where $u_{g} = u_1 \otimes \cdots \otimes u_N$ is the \defn{ground state element}, the unique element of $B$ such that $u_N = u_{\max}(B^{r_N,s_N})$ and $\varepsilon_i(u_j) = \varphi_i(u_{j+1})$ for all $1 \leq j < N$ and $i \in I$.
Moreover, we have
\[
B_v(\ell \Lambda_{\tau(0)}) \iso B \otimes u_{\ell \Lambda_0}\,,
\]
where $v\tau = t_{\clfw}$ with $v \in W$ and $\tau \in \Pi$.
\end{thm}

Theorem~\ref{thm:demazure_embedding} is conjectured to hold for all affine types~\cite{HKOTY99,HKOTT02,FL06} under the assumption that $s_j / c_{r_j} \in \ZZ$ implies that $B^{r_j,s_j}$ is perfect (cf.~the perfectness conjecture). We say that the above tensor product $B := \bigotimes_{j=1}^N B^{r_j, s_j}$ is of \defn{level bounded by $\ell$} if $\ell$ is such that $\lceil s_j / c_{r_j} \rceil \leq \ell$ for all $j$. Theorem~\ref{thm:demazure_embedding} was generalized in~\cite[Prop.~5.16]{Naoi13} in the following way: when $B$ (still of nonexceptional type) is a tensor product of level bounded by $\ell$, and $s_j / c_{r_j} \in \ZZ$, then $B \otimes u_{\ell \Lambda_0}$ is isomorphic to a direct sum of Demazure crystals.

We need the following fact from~\cite[Prop.~8.1]{ST12} (which was essentially proved in~\cite{KKMMNN91}). The claim holds for \emph{all affine types} since~\cite[Lemma~7.3]{ST12} holds, by using the general definition of energy~\cite{KKMMNN91,HKOTY99,HKOTT02}; see~\cite[Lemma~6.4]{Naoi13} as well.
We also note that there is \emph{no assumption of perfectness}.

\begin{prop}[{\cite[Prop.~8.1]{ST12},\cite{KKMMNN91}}]  
\label{prop:B_hw_decomp}
Consider a tensor product of KR crystals $B$ of level bounded by $\ell$. Then there exists a sequence $( \Lambda^{(k)} \in P_{\ell}^+ )_{k=1}^N$ such that
\[
B \otimes B(\ell \Lambda_0) \iso \bigoplus_{k=1}^N B\!\left(\Lambda^{(k)}\right)\,.
\]
\end{prop}

%=====================================================================

\section{Main results}
\label{sec:results}

For the remainder of this paper, we will consider KR crystals of nonexceptional type. It is known that, for any $J \subsetneq I$, under the Levi branching to the canonical subalgebra $\g_J$ with index set $J$ (\textit{i.e.}, we remove all $i$-edges for $i \in J \setminus I$), $B^{r,s}$ is a direct sum of highest weight $U_q(\g_J)$-crystals.

In this section we prove our main results.

\begin{remark}\label{contr-dual}
There are contragredient dual versions of all our results. All the proofs hold for the contragredient dual by interchanging $e_i \leftrightarrow f_i$.
\end{remark}

We start with our main tool, namely the generalization of Naoi's result~\cite[Prop.~5.16]{Naoi13} (which, in turn, is a generalization of Theorem~\ref{thm:demazure_embedding}, as discussed in Section~\ref{sec:kr}). Our generalization holds without the perfectness assumption. 

\begin{thm}
\label{thm:demazure_decomposition}
Let $B$ be a tensor product of KR crystals of level bounded by $\ell$, having the decomposition in Proposition~{\rm \ref{prop:B_hw_decomp}}.
Then there exists a sequence $( \lambda^{(k)} \in P_0^- )_{k=1}^N$ such that
\begin{equation}\label{dec-into-dem}
B \otimes u_{\ell \Lambda_0} = \bigoplus_{k=1}^N B_{\lambda^{(k)}} \iso \bigoplus_{k=1}^N B_{w^{(k)}}\!\left(\Lambda^{(k)}\right)\,;
\end{equation}
here the following hold:
\begin{itemize}
\item there exists a unique element $b^{(k)}_{\min} \otimes u_{\ell \Lambda_0} \in B_{\lambda^{(k)}}$ satisfying $\wt\!\left(b^{(k)}_{\min}\right) = \lambda^{(k)}$ and
\[
\wt(b) - \wt\!\left(b^{(k)}_{\min}\right) \in Q_0^+ \setminus \{0\} \;\;\:\mbox{ for all }\,b\otimes u_{\ell \Lambda_0} \in B_{\lambda^{(k)}} \setminus \{b^{(k)}_{\min} \otimes u_{\ell \Lambda_0} \}\,,
\]
\item $w^{(k)}\!\left(\Lambda^{(k)}\right) = \lambda^{(k)} + \ell \Lambda_0$, for $w^{(k)} \in W$ of minimal length.
\end{itemize}
\end{thm}

%For any $\mu \in P$ such that $\lev(\mu) > 0$, let $\mu^+$ denote the unique dominant weight in $W \mu$. 
Roughly speaking, our proof follows the proofs of~\cite[Prop.~5.16]{Naoi13} or~\cite[Thm.~4.7]{FSS07}, by reducing the statement to the case when $B = B^{r,s}$. Then we use the similarity map from Theorem~\ref{thm:KR_similarity}. In order to complete our proof, we use~\cite[Lemma~4.8]{Naoi13} and another elementary fact stated below.

\begin{lemma}[{\cite[Lemma~4.8]{Naoi13}}]
\label{lemma:Demazure_weights}
Let $\Lambda \in P^+$ and $w \in \widetilde{W}$, and assume that $\langle \alpha_i^{\vee}, w\Lambda \rangle \leq 0$ for all $i \in I_0$. Then for any $b \in B_w(\Lambda)$, we have $b = u_{w\Lambda}$ or
\[
\cl\bigl( \wt(b) \bigr) \in \cl( w \Lambda ) + (Q_0^+ \setminus \{0\})\,.
\]
%Moreover, $u_{w\Lambda}$ is an extremal element in $B_w(\Lambda)$.
\end{lemma}

\begin{lemma}
\label{lemma:Demazure_similarity}
Let $\sigma_m \colon B(\Lambda) \to B(m\Lambda)$ be a similarity map. For any $w \in W$, this induces a similarity map $\sigma_m^D \colon B_w(\Lambda) \to B_w(m\Lambda)$. The image of this map consists of those vertices $b$ in $B_w(m\Lambda)$ for which each $e_{i_j}$ in {\rm \eqref{def:dem}} is applied a multiple of $m$ number of times. 
\end{lemma}

\begin{proof}
This follows immediately from the definition of a similarity map~\eqref{def:sim} and of a Demazure crystal \eqref{def:dem}.
\end{proof}

\begin{proof}[Proof of Theorem {\rm \ref{thm:demazure_decomposition}}]
It is sufficient to restrict to $B = B^{r,s}$, by the reduction argument in the proof of~\cite[Prop.~5.16]{Naoi13} (this is essentially the induction described in~\cite[Thm.~4.7]{FSS07}).
Furthermore, it is sufficient to consider $B^{r,s} \otimes u_{\ell \Lambda_0}$, where $\ell = \lceil s / c_r \rceil$, by a similar argument to the one used in the proof of~\cite[Prop.~5.16]{Naoi13}.
When $s / c_r \in \ZZ$, the claim holds by Theorem~\ref{thm:demazure_embedding}. Therefore, assume $s / c_r \notin \ZZ$.

By Theorem~\ref{thm:demazure_embedding}, we have 
\[B^{r, c_r s} \otimes u_{s\Lambda_{\tau(0)}} \iso B_v(s \Lambda_{\tau(0)})\,,\]
 where $t_{-c_r \clfw_{r^*}} = v\tau$. Therefore, we have
\begin{align}\label{red-perf}
B^{r, c_r s} \otimes u_{c_r \ell \Lambda_{\tau(0)}} & \iso B^{r, c_r s} \otimes u_{s \Lambda_{\tau(0)}} \otimes u_{(c_r \ell - s) \Lambda_{\tau(0)}}
\\ & \iso B_v(s \Lambda_{\tau(0)}) \otimes u_{(c_r \ell - s) \Lambda_{\tau(0)}} \iso \bigoplus_{(w,\Lambda)} B_w(\Lambda)\,,\nonumber
\end{align}
where the last isomorphism is by the combinatorial excellent filtration in Theorem~\ref{thm:excellent_filtration}. 
Now consider the similarity map $\sigma_{c_r}: B^{r,s} \to B^{r, c_r s}$ from Theorem~\ref{thm:KR_similarity}.
This gives a similarity map 
\[\sigma^D_{c_r} \colon B^{r, s} \otimes u_{\ell \Lambda_{\tau(0)}} \to B^{r,c_r s} \otimes u_{c_r\ell \Lambda_{\tau(0)}}\,.\]
Composing the latter map with the crystal isomorphisms in \eqref{red-perf}, we want to identify the image $\im \sigma^D_{c_r}$ of $B^{r, s} \otimes u_{\ell \Lambda_{\tau(0)}}$ inside the direct sum of Demazure crystals in \eqref{red-perf}. 

Now assume that $\im \sigma^D_{c_r}$ intersects some Demazure crystal $B_w(\Lambda)$. Pick some vertex in the intersection, and a sequence of $e_{i_j}$ to $u_{\Lambda}$ as in \eqref{def:dem}. By the definition of a similarity map \eqref{def:sim}, each $e_{i_j}$ is applied a multiple of $k$ number of times, and the upper endpoints of the various strings belong to $\im \sigma^D_{c_r}\cap B_w(\Lambda)$. In particular, so does $u_{\Lambda}$.  We can now see that $\im \sigma^D_{c_r}\cap B_w(\Lambda)$ is characterized by the condition in Lemma~\ref{lemma:Demazure_similarity}. 

As $\wt\bigl( \sigma^D_{c_r}(b) \bigr) = c_r \wt(b)$ for all $b \in B^{r,s} \otimes u_{\ell \Lambda_{\tau(0)}}$, by~\eqref{def:sim}, we have $\Lambda / c_r \in P^+$. Combining the above facts with Lemma~\ref{lemma:Demazure_similarity}, we deduce that
\[\im \sigma^D_{c_r}\cap B_w(\Lambda) \iso B_w(\Lambda / c_r)\,.\]
Thus, we proved the decomposition~\eqref{dec-into-dem}.

Note that the multiset of weights $\Lambda^{(k)}$ in \eqref{dec-into-dem} coincides with the one in Proposition~{\rm \ref{prop:B_hw_decomp}}. Indeed, by the signature rule, all highest weight vertices in $B \otimes B(\ell \Lambda_0)$ are of the form $b\otimes u_{\ell \Lambda_0}$ with $b\in B$, so they are highest weight vertices in $B \otimes u_{\ell \Lambda_0}$; the reverse inclusion is obvious. 

Finally, the existence of the vertices $b_{\min}^{(k)}$ with the desired properties follows from Lemma~\ref{lemma:Demazure_weights}. Indeed, let us verify the hypothesis of this Lemma. Take an affine Demazure crystal $B_{w^{(k)}}\!\left(\Lambda^{(k)}\right)$ in~\eqref{dec-into-dem}. It has a decomposition into classically highest weight crystals, because $B$ has such a decomposition, and tensoring with $u_{\ell \Lambda_0}$ does not affect the classical crystal structure. In the mentioned decomposition of $B_{w^{(k)}}\!\left(\Lambda^{(k)}\right)$, the unique element of weight $w^{(k)}\!\left(\Lambda^{(k)}\right)$ has to be a classically lowest weight element, so $w^{(k)}\!\left(\Lambda^{(k)}\right)-\ell \Lambda_0$ is a finite antidominant weight. 
\end{proof}

\begin{remarks}
(1) We could have proved Theorem~\ref{thm:demazure_decomposition} by applying the similarity map and then by directly appealing to~\cite[Prop.~5.16]{Naoi13}. However, we found it more illuminating to show the role of the similarity map in a very explicit way: in the single factor case.

(2) Theorem~\ref{thm:demazure_decomposition} does not imply that $B \otimes u_{\ell \Lambda_0}$ is a single Demazure crystal. However, it has been shown in~\cite{FL06} that $B = B^{r,s}$ is isomorphic as an $U_q(\g_0)$-crystal to a single Demazure crystal for a number of special $r \in I_0$ in exceptional types. If we combine this with Theorem~\ref{thm:demazure_decomposition}, we obtain Theorem~\ref{thm:demazure_embedding} for these cases.
\end{remarks}

Following~\cite{ST12}, we define a \defn{level $\ell$ Demazure edge} as being an $i$-edge $b' \to b$ in the crystal graph such that either $i \in I_0$ or $\varepsilon_0(b) > \ell$. In other words, an edge is not a level $\ell$ Demazure edge if it is in the length $\ell$ head of a $0$-string. A  \defn{level $\ell$ dual Demazure edge} is defined similarly, using the length $\ell$ tail of a $0$-string. 
Let $\widetilde{D}_{\ell}(B)$ and $\widetilde{DD}_{\ell}(B)$ denote the subcrystals of $B$ obtained by removing all edges that are not level $\ell$ Demazure edges in $B$, respectively dual Demazure edges.

The following lemma is well-known to experts. It follows immediately from the tensor product rule, and motivates the terminology of Demazure edge.

\begin{lemma}
\label{lemma:demazure_removal}
Let $B$ be a tensor product of KR crystals of level bounded by $\ell$.
The map
\begin{align*}
\rho_{\ell} \colon & \widetilde{D}_{\ell}(B) \to B \otimes u_{\ell \Lambda_0}
\end{align*}
given by $\rho_{\ell}(b) = b \otimes u_{\ell \Lambda_0}$ is a crystal isomorphism (up to a weight shift).
\end{lemma}

\begin{remark}
\label{rem:grading_energy}
Let $\psi \colon \bigoplus_{k=1}^N B_{\lambda^{(k)}} \to B \otimes u_{\ell \Lambda_0}$ be the isomorphism given by Theorem~\ref{thm:demazure_embedding}.
By~\cite[Lemma~7.3]{ST12} and~\cite[Lemma~6.4]{Naoi13}, then there exist constants $(C_k)_{k=1}^N$ such that for all $b \in B_{\lambda^{(k)}}$, if $\wt(b) = \mu + D \delta$ and $\psi(b) = b' \otimes u_{\ell \Lambda_0}$ (note we are considering this as a $U_q(\g)$-crystal and have to implicitly branch to $U_q'(\g)$, which simply changes the weight), then we have $D = E(b') + C_k$, where $E(b)$ is the energy statistic of~\cite{KKMMNN91,HKOTY99,HKOTT02}. These constants were explicitly specified when $B$ was a tensor product of perfect crystals in nonexceptional type in~\cite[Thm.~7.1]{Naoi13}.
\end{remark}

Let $D_{\ell}(B)$ denote the connected component of $\widetilde{D}_{\ell}(B)$ that contains $u_{\min}(B)$, and $DD_{\ell}(B)$ denote the connected component of $\widetilde{DD}_{\ell}(B)$ that contains $u_{\max}(B)$. We now present our main theorem.

\begin{thm}\label{thm:red}
Let $B := \bigotimes_{j=1}^N B^{r_j,s_j}$ and $B' := \bigotimes_{j=1}^{N'} B^{r'_j,s'_j}$ be of levels bounded by $\ell$ and
\begin{equation}
\label{eq:max_weight}
\sum_{j=1}^N s_j \clfw_{r_j} = \sum_{j=1}^{N'} s'_j \clfw_{r'_j}\,.
\end{equation}
Then we have
\[
D_{\ell}(B) \iso D_{\ell}(B')\,,\;\;\;\;\;\;\;\;\;{DD}_{\ell}(B) \iso DD_{\ell}(B')\,.
\]
\end{thm}

\begin{proof}
Let $\lambda$ be the weight given by~\eqref{eq:max_weight}, and let $\Lambda := w_0(\lambda)+\ell\Lambda_0$.
Note that $\wt\bigl( u_{\min}(B) \bigr) = w_0(\lambda)$.
Since we have $\wt(b) \in w_0(\lambda) + (Q_0^+ \setminus \{0\})$ for all $b \in B \setminus \{u_{\min}(B)\}$, we deduce that $D_{\ell}(B) \iso B_w(\mu)$ for some $\mu \in P_{\ell}^+$ and $w \in W$, by Theorem~\ref{thm:demazure_decomposition} and Lemma~\ref{lemma:demazure_removal}; here $w$ and $\mu$ are uniquely determined by the condition $w(\mu)=\Lambda$ and the fact that $w$ is of minimum length. 
Similarly, we have $D_{\ell}(B') \iso B_w(\mu)$.
Hence, we have $D_{\ell}(B) \iso D_{\ell}(B')$. The proof of the contragredient dual version is completely similar, cf. Remark~\ref{contr-dual}.
\end{proof}

We conjecture that Theorems~\ref{thm:demazure_decomposition} and~\ref{thm:red} also hold in the exceptional types. As evidence, we note that Theorem~\ref{thm:demazure_embedding} is known to hold as $U_q(\g_0)$-crystals in certain cases by~\cite{FL06}, which hence implies $U_q(\g_0)$-crystal versions of Theorems~\ref{thm:demazure_decomposition} and~\ref{thm:red}. Furthermore, Theorem~\ref{thm:demazure_decomposition} is known to hold for $\ell = 1$ in all types~\cite{K05,FL07,Naoi12}.

%=====================================================================

\section{Uniform models}\label{models}

\subsection{Reduction to single-column KR crystals} The following corollary of Theorem~\ref{thm:red} allows us to reduce tensor products of arbitrary KR crystals to tensor products of single-column ones. 

\begin{cor}
\label{cor:red}
Let $B := \bigotimes_{j=1}^N B^{r_j,s_j}$ and $B' := \bigotimes_{j=1}^{N} (B^{r_j,1})^{\otimes s_j}$ be such that there exists $\ell\in\ZZ$ with $s_j / c_{r_j}=\ell$ for all $j$.
Then we have 
\[
\widetilde{D}_{\ell}(B) \iso D_{\ell}(B')\,,\;\;\;\;\;\;\;\;\;\widetilde{DD}_{\ell}(B) \iso DD_{\ell}(B')\,.
\]
\end{cor}

\begin{proof}
Since $s_j/c_{r_j}=\ell\in\ZZ$, all $B^{r_j,s_j}$ are perfect crystals of level $\ell$. This property implies that $\widetilde{D}_{\ell}(B)$ is connected, so $\widetilde{D}_{\ell}(B)=D_{\ell}(B)$, and similarly for the contragredient dual case. Now apply Theorem~\ref{thm:red}. 
\end{proof}

\begin{remarks}\label{rem:red}
\mbox{}
\begin{itemize}
\item[(1)]  The isomorphism in Corollary~\ref{cor:red} realizes the level $\ell$ Demazure and dual Demazure portions of $B$ in terms of single-column KR crystals. 

\item[(2)] We conjecture a similar realization in the exceptional types, assuming the perfectness conjecture. We expect the proof to be completely similar, based on the generalizations of Theorems~\ref{thm:demazure_decomposition} and~\ref{thm:red} that were conjectured in Section~\ref{sec:results}. We have verified this conjecture for $B^{r,s}$ for $s = 2,3,4$ in types $D_4^{(3)}$ (which is perfect for all $s$) and $G_2^{(1)}$. (These cases could follow by a similar proof of Theorem~\ref{thm:demazure_decomposition} using a diagram folding of $D_4^{(1)}$ and the corresponding conjectural virtual crystal construction; see, \textit{e.g.},~\cite{OSS03II,PS15,SchillingS15}.)
\end{itemize}
\end{remarks}

The following natural question arises.

\begin{problem}
\label{prob:isomorphism}
How is the isomorphism in Corollary~\ref{cor:red} expressed concretely when the corresponding tensor products of KR crystals are realized based on the tableau model~\cite{FOS09} and the rigged configuration model~\cite{SchillingS15}?
\end{problem}

One particular approach to Problem~\ref{prob:isomorphism} could be through the use of the so-called \defn{Kirillov--Reshetikhin (KR) tableaux} of~\cite{OSS13,SchillingS15}. These arise from the bijection $\Phi$ with rigged configurations, which use column splitting to construct classical crystal embeddings as a core part of the bijection $\Phi$. Considering the fact that rigged configurations and the Demazure constructions are combinatorial $R$-matrix invariant, as well as the relationship with energy and the affine grading from~\cite{FSS07,KKMMNN92,ST12} (see also Remark~\ref{rem:grading_energy}), it is likely that $\Phi$, and hence KR tableaux, could be a consequence of Theorem~\ref{thm:red}.

\begin{ex}\label{ex-typeC}
In type $C_2^{(1)}$, we have two connected components for $\widetilde{D}_1(B^{1,1} \otimes B^{1,1})$, and a single connected component for $\widetilde{D}_1(B^{1,2})$; see Figure~\ref{fig:C2_ex}, where the crystal vertices are labeled by the corresponding \defn{Kashiwara--Nakashima tableaux}, see~\cite{BS17}. Note that the leftmost connected component of the former crystal is $D_1(B^{1,1} \otimes B^{1,1})$, and this is isomorphic to $D_1(B^{1,2}) = \widetilde{D}_1(B^{1,2})$. The latter is thus realized in terms of single-column KR crystals.

\def\boxelt#1{\begin{array}{|@{\hspace{.6ex}}c@{\hspace{.6ex}}|}\hline \raisebox{-.3ex}{$#1$} \\\hline \end{array}}
Note that the corresponding element of $\emptyset \in \widetilde{D}_1(B^{1,2})$ corresponds to $\boxelt{\bon} \otimes \boxelt{1} \in \widetilde{D}_1(B^{1,1} \otimes B^{1,1})$, which is the splitting of the corresponding KR tableaux in $B^{1,2}$ (the other elements are also Kashiwara--Nakashima tableaux).
\end{ex}

\begin{figure}
\[
\begin{tikzpicture}[>=latex,line join=bevel,yscale=0.7,baseline=0]
\node (node_13) at (133.5bp,228.5bp) [draw,draw=none] {${\def\lr#1{\multicolumn{1}{|@{\hspace{.6ex}}c@{\hspace{.6ex}}|}{\raisebox{-.3ex}{$#1$}}}\raisebox{-.6ex}{$\begin{array}[b]{*{1}c}\cline{1-1}\lr{\overline{1}}\\\cline{1-1}\end{array}$}} \otimes {\def\lr#1{\multicolumn{1}{|@{\hspace{.6ex}}c@{\hspace{.6ex}}|}{\raisebox{-.3ex}{$#1$}}}\raisebox{-.6ex}{$\begin{array}[b]{*{1}c}\cline{1-1}\lr{2}\\\cline{1-1}\end{array}$}}$};
  \node (node_4) at (133.5bp,447.5bp) [draw,draw=none] {${\def\lr#1{\multicolumn{1}{|@{\hspace{.6ex}}c@{\hspace{.6ex}}|}{\raisebox{-.3ex}{$#1$}}}\raisebox{-.6ex}{$\begin{array}[b]{*{1}c}\cline{1-1}\lr{2}\\\cline{1-1}\end{array}$}} \otimes {\def\lr#1{\multicolumn{1}{|@{\hspace{.6ex}}c@{\hspace{.6ex}}|}{\raisebox{-.3ex}{$#1$}}}\raisebox{-.6ex}{$\begin{array}[b]{*{1}c}\cline{1-1}\lr{1}\\\cline{1-1}\end{array}$}}$};
  \node (node_9) at (133.5bp,301.5bp) [draw,draw=none] {${\def\lr#1{\multicolumn{1}{|@{\hspace{.6ex}}c@{\hspace{.6ex}}|}{\raisebox{-.3ex}{$#1$}}}\raisebox{-.6ex}{$\begin{array}[b]{*{1}c}\cline{1-1}\lr{\overline{2}}\\\cline{1-1}\end{array}$}} \otimes {\def\lr#1{\multicolumn{1}{|@{\hspace{.6ex}}c@{\hspace{.6ex}}|}{\raisebox{-.3ex}{$#1$}}}\raisebox{-.6ex}{$\begin{array}[b]{*{1}c}\cline{1-1}\lr{2}\\\cline{1-1}\end{array}$}}$};
  \node (node_8) at (133.5bp,374.5bp) [draw,draw=none] {${\def\lr#1{\multicolumn{1}{|@{\hspace{.6ex}}c@{\hspace{.6ex}}|}{\raisebox{-.3ex}{$#1$}}}\raisebox{-.6ex}{$\begin{array}[b]{*{1}c}\cline{1-1}\lr{\overline{2}}\\\cline{1-1}\end{array}$}} \otimes {\def\lr#1{\multicolumn{1}{|@{\hspace{.6ex}}c@{\hspace{.6ex}}|}{\raisebox{-.3ex}{$#1$}}}\raisebox{-.6ex}{$\begin{array}[b]{*{1}c}\cline{1-1}\lr{1}\\\cline{1-1}\end{array}$}}$};
  \node (node_7) at (19.5bp,155.5bp) [draw,draw=none] {${\def\lr#1{\multicolumn{1}{|@{\hspace{.6ex}}c@{\hspace{.6ex}}|}{\raisebox{-.3ex}{$#1$}}}\raisebox{-.6ex}{$\begin{array}[b]{*{1}c}\cline{1-1}\lr{2}\\\cline{1-1}\end{array}$}} \otimes {\def\lr#1{\multicolumn{1}{|@{\hspace{.6ex}}c@{\hspace{.6ex}}|}{\raisebox{-.3ex}{$#1$}}}\raisebox{-.6ex}{$\begin{array}[b]{*{1}c}\cline{1-1}\lr{\overline{1}}\\\cline{1-1}\end{array}$}}$};
  \node (node_6) at (76.5bp,228.5bp) [draw,draw=none] {${\def\lr#1{\multicolumn{1}{|@{\hspace{.6ex}}c@{\hspace{.6ex}}|}{\raisebox{-.3ex}{$#1$}}}\raisebox{-.6ex}{$\begin{array}[b]{*{1}c}\cline{1-1}\lr{2}\\\cline{1-1}\end{array}$}} \otimes {\def\lr#1{\multicolumn{1}{|@{\hspace{.6ex}}c@{\hspace{.6ex}}|}{\raisebox{-.3ex}{$#1$}}}\raisebox{-.6ex}{$\begin{array}[b]{*{1}c}\cline{1-1}\lr{\overline{2}}\\\cline{1-1}\end{array}$}}$};
  \node (node_5) at (76.5bp,301.5bp) [draw,draw=none] {${\def\lr#1{\multicolumn{1}{|@{\hspace{.6ex}}c@{\hspace{.6ex}}|}{\raisebox{-.3ex}{$#1$}}}\raisebox{-.6ex}{$\begin{array}[b]{*{1}c}\cline{1-1}\lr{2}\\\cline{1-1}\end{array}$}} \otimes {\def\lr#1{\multicolumn{1}{|@{\hspace{.6ex}}c@{\hspace{.6ex}}|}{\raisebox{-.3ex}{$#1$}}}\raisebox{-.6ex}{$\begin{array}[b]{*{1}c}\cline{1-1}\lr{2}\\\cline{1-1}\end{array}$}}$};
  \node (node_14) at (133.5bp,155.5bp) [draw,draw=none] {${\def\lr#1{\multicolumn{1}{|@{\hspace{.6ex}}c@{\hspace{.6ex}}|}{\raisebox{-.3ex}{$#1$}}}\raisebox{-.6ex}{$\begin{array}[b]{*{1}c}\cline{1-1}\lr{\overline{1}}\\\cline{1-1}\end{array}$}} \otimes {\def\lr#1{\multicolumn{1}{|@{\hspace{.6ex}}c@{\hspace{.6ex}}|}{\raisebox{-.3ex}{$#1$}}}\raisebox{-.6ex}{$\begin{array}[b]{*{1}c}\cline{1-1}\lr{\overline{2}}\\\cline{1-1}\end{array}$}}$};
  \node (node_3) at (19.5bp,228.5bp) [draw,draw=none] {${\def\lr#1{\multicolumn{1}{|@{\hspace{.6ex}}c@{\hspace{.6ex}}|}{\raisebox{-.3ex}{$#1$}}}\raisebox{-.6ex}{$\begin{array}[b]{*{1}c}\cline{1-1}\lr{1}\\\cline{1-1}\end{array}$}} \otimes {\def\lr#1{\multicolumn{1}{|@{\hspace{.6ex}}c@{\hspace{.6ex}}|}{\raisebox{-.3ex}{$#1$}}}\raisebox{-.6ex}{$\begin{array}[b]{*{1}c}\cline{1-1}\lr{\overline{1}}\\\cline{1-1}\end{array}$}}$};
  \node (node_2) at (19.5bp,301.5bp) [draw,draw=none] {${\def\lr#1{\multicolumn{1}{|@{\hspace{.6ex}}c@{\hspace{.6ex}}|}{\raisebox{-.3ex}{$#1$}}}\raisebox{-.6ex}{$\begin{array}[b]{*{1}c}\cline{1-1}\lr{1}\\\cline{1-1}\end{array}$}} \otimes {\def\lr#1{\multicolumn{1}{|@{\hspace{.6ex}}c@{\hspace{.6ex}}|}{\raisebox{-.3ex}{$#1$}}}\raisebox{-.6ex}{$\begin{array}[b]{*{1}c}\cline{1-1}\lr{\overline{2}}\\\cline{1-1}\end{array}$}}$};
  \node (node_1) at (19.5bp,374.5bp) [draw,draw=none] {${\def\lr#1{\multicolumn{1}{|@{\hspace{.6ex}}c@{\hspace{.6ex}}|}{\raisebox{-.3ex}{$#1$}}}\raisebox{-.6ex}{$\begin{array}[b]{*{1}c}\cline{1-1}\lr{1}\\\cline{1-1}\end{array}$}} \otimes {\def\lr#1{\multicolumn{1}{|@{\hspace{.6ex}}c@{\hspace{.6ex}}|}{\raisebox{-.3ex}{$#1$}}}\raisebox{-.6ex}{$\begin{array}[b]{*{1}c}\cline{1-1}\lr{2}\\\cline{1-1}\end{array}$}}$};
  \node (node_10) at (76.5bp,155.5bp) [draw,draw=none] {${\def\lr#1{\multicolumn{1}{|@{\hspace{.6ex}}c@{\hspace{.6ex}}|}{\raisebox{-.3ex}{$#1$}}}\raisebox{-.6ex}{$\begin{array}[b]{*{1}c}\cline{1-1}\lr{\overline{2}}\\\cline{1-1}\end{array}$}} \otimes {\def\lr#1{\multicolumn{1}{|@{\hspace{.6ex}}c@{\hspace{.6ex}}|}{\raisebox{-.3ex}{$#1$}}}\raisebox{-.6ex}{$\begin{array}[b]{*{1}c}\cline{1-1}\lr{\overline{2}}\\\cline{1-1}\end{array}$}}$};
  \node (node_11) at (47.5bp,82.5bp) [draw,draw=none] {${\def\lr#1{\multicolumn{1}{|@{\hspace{.6ex}}c@{\hspace{.6ex}}|}{\raisebox{-.3ex}{$#1$}}}\raisebox{-.6ex}{$\begin{array}[b]{*{1}c}\cline{1-1}\lr{\overline{2}}\\\cline{1-1}\end{array}$}} \otimes {\def\lr#1{\multicolumn{1}{|@{\hspace{.6ex}}c@{\hspace{.6ex}}|}{\raisebox{-.3ex}{$#1$}}}\raisebox{-.6ex}{$\begin{array}[b]{*{1}c}\cline{1-1}\lr{\overline{1}}\\\cline{1-1}\end{array}$}}$};
  \node (node_0) at (48.5bp,447.5bp) [draw,draw=none] {${\def\lr#1{\multicolumn{1}{|@{\hspace{.6ex}}c@{\hspace{.6ex}}|}{\raisebox{-.3ex}{$#1$}}}\raisebox{-.6ex}{$\begin{array}[b]{*{1}c}\cline{1-1}\lr{1}\\\cline{1-1}\end{array}$}} \otimes {\def\lr#1{\multicolumn{1}{|@{\hspace{.6ex}}c@{\hspace{.6ex}}|}{\raisebox{-.3ex}{$#1$}}}\raisebox{-.6ex}{$\begin{array}[b]{*{1}c}\cline{1-1}\lr{1}\\\cline{1-1}\end{array}$}}$};
  \node (node_15) at (47.5bp,9.5bp) [draw,draw=none] {${\def\lr#1{\multicolumn{1}{|@{\hspace{.6ex}}c@{\hspace{.6ex}}|}{\raisebox{-.3ex}{$#1$}}}\raisebox{-.6ex}{$\begin{array}[b]{*{1}c}\cline{1-1}\lr{\overline{1}}\\\cline{1-1}\end{array}$}} \otimes {\def\lr#1{\multicolumn{1}{|@{\hspace{.6ex}}c@{\hspace{.6ex}}|}{\raisebox{-.3ex}{$#1$}}}\raisebox{-.6ex}{$\begin{array}[b]{*{1}c}\cline{1-1}\lr{\overline{1}}\\\cline{1-1}\end{array}$}}$};
  \node (node_12) at (76.5bp,374.5bp) [draw,draw=none] {${\def\lr#1{\multicolumn{1}{|@{\hspace{.6ex}}c@{\hspace{.6ex}}|}{\raisebox{-.3ex}{$#1$}}}\raisebox{-.6ex}{$\begin{array}[b]{*{1}c}\cline{1-1}\lr{\overline{1}}\\\cline{1-1}\end{array}$}} \otimes {\def\lr#1{\multicolumn{1}{|@{\hspace{.6ex}}c@{\hspace{.6ex}}|}{\raisebox{-.3ex}{$#1$}}}\raisebox{-.6ex}{$\begin{array}[b]{*{1}c}\cline{1-1}\lr{1}\\\cline{1-1}\end{array}$}}$};
  \draw [red,->] (node_7) ..controls (23.781bp,136.53bp) and (27.579bp,122.01bp)  .. (32.5bp,110.0bp) .. controls (33.718bp,107.03bp) and (35.173bp,103.96bp)  .. (node_11);
  \definecolor{strokecol}{rgb}{0.0,0.0,0.0};
  \pgfsetstrokecolor{strokecol}
  \draw (41.0bp,119.0bp) node {$2$};
  \draw [blue,->] (node_2) ..controls (19.5bp,280.98bp) and (19.5bp,262.65bp)  .. (node_3);
  \draw (28.0bp,265.0bp) node {$1$};
  \draw [blue,->] (node_1) ..controls (35.948bp,353.43bp) and (51.267bp,333.82bp)  .. (node_5);
  \draw (64.0bp,338.0bp) node {$1$};
  \draw [red,->] (node_6) ..controls (76.5bp,207.98bp) and (76.5bp,189.65bp)  .. (node_10);
  \draw (85.0bp,192.0bp) node {$2$};
  \draw [blue,->] (node_8) ..controls (133.5bp,353.98bp) and (133.5bp,335.65bp)  .. (node_9);
  \draw (142.0bp,338.0bp) node {$1$};
  \draw [red,->] (node_1) ..controls (19.5bp,353.98bp) and (19.5bp,335.65bp)  .. (node_2);
  \draw (28.0bp,338.0bp) node {$2$};
  \draw [red,->] (node_4) ..controls (133.5bp,426.98bp) and (133.5bp,408.65bp)  .. (node_8);
  \draw (142.0bp,411.0bp) node {$2$};
  \draw [black,<-] (node_0) ..controls (61.32bp,414.08bp) and (68.456bp,395.47bp)  .. (node_12);
  \draw (75.0bp,411.0bp) node {$0$};
  \draw [red,->] (node_13) ..controls (133.5bp,207.98bp) and (133.5bp,189.65bp)  .. (node_14);
  \draw (142.0bp,192.0bp) node {$2$};
  \draw [blue,->] (node_11) ..controls (47.5bp,61.978bp) and (47.5bp,43.649bp)  .. (node_15);
  \draw (56.0bp,46.0bp) node {$1$};
  \draw [blue,->] (node_9) ..controls (133.5bp,280.98bp) and (133.5bp,262.65bp)  .. (node_13);
  \draw (142.0bp,265.0bp) node {$1$};
  \draw [blue,->] (node_0) ..controls (40.261bp,426.76bp) and (32.776bp,407.92bp)  .. (node_1);
  \draw (46.0bp,411.0bp) node {$1$};
  \draw [blue,->] (node_10) ..controls (68.261bp,134.76bp) and (60.776bp,115.92bp)  .. (node_11);
  \draw (74.0bp,119.0bp) node {$1$};
  \draw [red,->] (node_5) ..controls (76.5bp,280.98bp) and (76.5bp,262.65bp)  .. (node_6);
  \draw (85.0bp,265.0bp) node {$2$};
  \draw [blue,->] (node_3) ..controls (19.5bp,207.98bp) and (19.5bp,189.65bp)  .. (node_7);
  \draw (28.0bp,192.0bp) node {$1$};
\end{tikzpicture}
\hspace{70pt}
\begin{tikzpicture}[>=latex,line join=bevel,yscale=0.7,baseline=0]
\node (node_9) at (35.5bp,82.5bp) [draw,draw=none] {${\def\lr#1{\multicolumn{1}{|@{\hspace{.6ex}}c@{\hspace{.6ex}}|}{\raisebox{-.3ex}{$#1$}}}\raisebox{-.6ex}{$\begin{array}[b]{*{2}c}\cline{1-2}\lr{\overline{2}}&\lr{\overline{1}}\\\cline{1-2}\end{array}$}}$};
  \node (node_8) at (13.5bp,155.5bp) [draw,draw=none] {${\def\lr#1{\multicolumn{1}{|@{\hspace{.6ex}}c@{\hspace{.6ex}}|}{\raisebox{-.3ex}{$#1$}}}\raisebox{-.6ex}{$\begin{array}[b]{*{2}c}\cline{1-2}\lr{\overline{2}}&\lr{\overline{2}}\\\cline{1-2}\end{array}$}}$};
  \node (node_7) at (58.5bp,155.5bp) [draw,draw=none] {${\def\lr#1{\multicolumn{1}{|@{\hspace{.6ex}}c@{\hspace{.6ex}}|}{\raisebox{-.3ex}{$#1$}}}\raisebox{-.6ex}{$\begin{array}[b]{*{2}c}\cline{1-2}\lr{2}&\lr{\overline{1}}\\\cline{1-2}\end{array}$}}$};
  \node (node_6) at (13.5bp,228.5bp) [draw,draw=none] {${\def\lr#1{\multicolumn{1}{|@{\hspace{.6ex}}c@{\hspace{.6ex}}|}{\raisebox{-.3ex}{$#1$}}}\raisebox{-.6ex}{$\begin{array}[b]{*{2}c}\cline{1-2}\lr{2}&\lr{\overline{2}}\\\cline{1-2}\end{array}$}}$};
  \node (node_5) at (13.5bp,301.5bp) [draw,draw=none] {${\def\lr#1{\multicolumn{1}{|@{\hspace{.6ex}}c@{\hspace{.6ex}}|}{\raisebox{-.3ex}{$#1$}}}\raisebox{-.6ex}{$\begin{array}[b]{*{2}c}\cline{1-2}\lr{2}&\lr{2}\\\cline{1-2}\end{array}$}}$};
  \node (node_4) at (58.5bp,228.5bp) [draw,draw=none] {${\def\lr#1{\multicolumn{1}{|@{\hspace{.6ex}}c@{\hspace{.6ex}}|}{\raisebox{-.3ex}{$#1$}}}\raisebox{-.6ex}{$\begin{array}[b]{*{2}c}\cline{1-2}\lr{1}&\lr{\overline{1}}\\\cline{1-2}\end{array}$}}$};
  \node (node_3) at (58.5bp,301.5bp) [draw,draw=none] {${\def\lr#1{\multicolumn{1}{|@{\hspace{.6ex}}c@{\hspace{.6ex}}|}{\raisebox{-.3ex}{$#1$}}}\raisebox{-.6ex}{$\begin{array}[b]{*{2}c}\cline{1-2}\lr{1}&\lr{\overline{2}}\\\cline{1-2}\end{array}$}}$};
  \node (node_2) at (35.5bp,374.5bp) [draw,draw=none] {${\def\lr#1{\multicolumn{1}{|@{\hspace{.6ex}}c@{\hspace{.6ex}}|}{\raisebox{-.3ex}{$#1$}}}\raisebox{-.6ex}{$\begin{array}[b]{*{2}c}\cline{1-2}\lr{1}&\lr{2}\\\cline{1-2}\end{array}$}}$};
  \node (node_1) at (55.5bp,447.5bp) [draw,draw=none] {${\def\lr#1{\multicolumn{1}{|@{\hspace{.6ex}}c@{\hspace{.6ex}}|}{\raisebox{-.3ex}{$#1$}}}\raisebox{-.6ex}{$\begin{array}[b]{*{2}c}\cline{1-2}\lr{1}&\lr{1}\\\cline{1-2}\end{array}$}}$};
  \node (node_10) at (35.5bp,9.5bp) [draw,draw=none] {${\def\lr#1{\multicolumn{1}{|@{\hspace{.6ex}}c@{\hspace{.6ex}}|}{\raisebox{-.3ex}{$#1$}}}\raisebox{-.6ex}{$\begin{array}[b]{*{2}c}\cline{1-2}\lr{\overline{1}}&\lr{\overline{1}}\\\cline{1-2}\end{array}$}}$};
  \node (node_0) at (75.5bp,374.5bp) [draw,draw=none] {${\emptyset}$};
  \draw [red,->] (node_2) ..controls (42.0bp,353.87bp) and (47.855bp,335.29bp)  .. (node_3);
  \definecolor{strokecol}{rgb}{0.0,0.0,0.0};
  \pgfsetstrokecolor{strokecol}
  \draw (58.0bp,338.0bp) node {$2$};
  \draw [blue,->] (node_4) ..controls (58.5bp,207.98bp) and (58.5bp,189.65bp)  .. (node_7);
  \draw (67.0bp,192.0bp) node {$1$};
  \draw [blue,->] (node_9) ..controls (35.5bp,61.978bp) and (35.5bp,43.649bp)  .. (node_10);
  \draw (44.0bp,46.0bp) node {$1$};
  \draw [blue,->] (node_8) ..controls (12.835bp,136.4bp) and (13.397bp,121.84bp)  .. (17.5bp,110.0bp) .. controls (18.621bp,106.77bp) and (20.206bp,103.54bp)  .. (node_9);
  \draw (26.0bp,119.0bp) node {$1$};
  \draw [blue,->] (node_1) ..controls (46.708bp,432.78bp) and (43.438bp,426.24bp)  .. (41.5bp,420.0bp) .. controls (38.934bp,411.73bp) and (37.465bp,402.29bp)  .. (node_2);
  \draw (50.0bp,411.0bp) node {$1$};
  \draw [black,<-] (node_1) ..controls (64.967bp,412.95bp) and (70.482bp,392.82bp)  .. (node_0);
  \draw (76.0bp,411.0bp) node {$0$};
  \draw [blue,->] (node_3) ..controls (58.5bp,280.98bp) and (58.5bp,262.65bp)  .. (node_4);
  \draw (67.0bp,265.0bp) node {$1$};
  \draw [red,->] (node_7) ..controls (52.0bp,134.87bp) and (46.145bp,116.29bp)  .. (node_9);
  \draw (58.0bp,119.0bp) node {$2$};
  \draw [blue,->] (node_2) ..controls (23.838bp,359.92bp) and (19.742bp,353.47bp)  .. (17.5bp,347.0bp) .. controls (14.679bp,338.86bp) and (13.532bp,329.44bp)  .. (node_5);
  \draw (26.0bp,338.0bp) node {$1$};
  \draw [red,->] (node_5) ..controls (13.5bp,280.98bp) and (13.5bp,262.65bp)  .. (node_6);
  \draw (22.0bp,265.0bp) node {$2$};
  \draw [red,->] (node_6) ..controls (13.5bp,207.98bp) and (13.5bp,189.65bp)  .. (node_8);
  \draw (22.0bp,192.0bp) node {$2$};
\end{tikzpicture}
\]
\caption{The crystals $\widetilde{D}_1(B^{1,1} \otimes B^{1,1})$ (left) and $\widetilde{D}_1(B^{1,2})$ (right) in Example~\ref{ex-typeC}.}
\label{fig:C2_ex}
\end{figure}
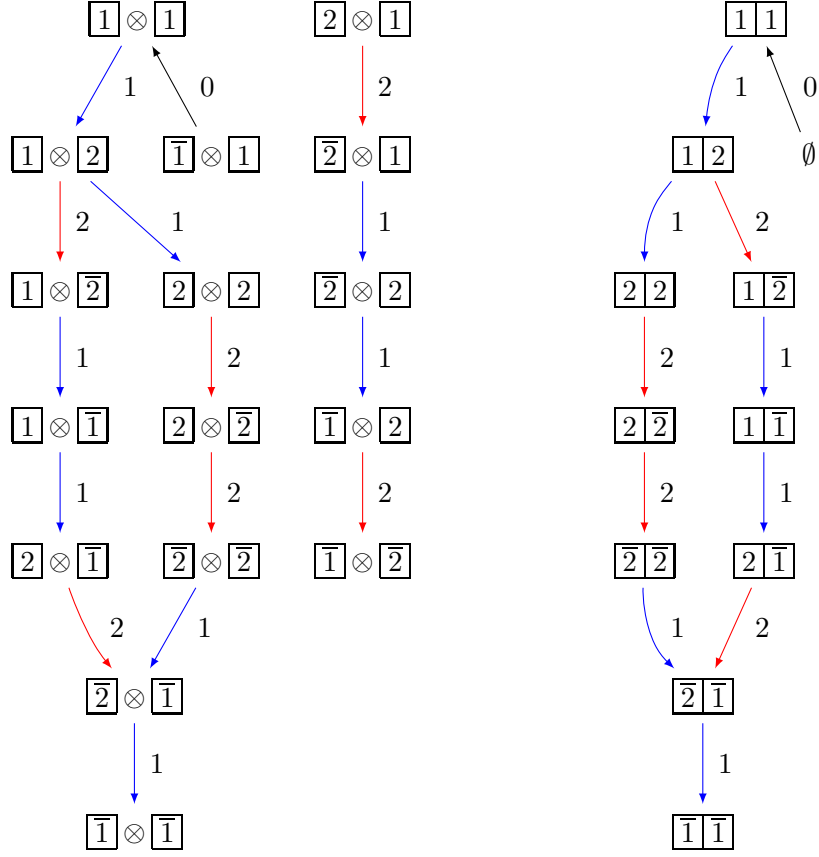

\subsection{The quantum alcove model}

We now recall the quantum alcove model and the main results related to it. For more details, including examples, we refer to the relevant papers \cite{LL15,LNSSS14II,LL18}. The setup is that of a finite root system $\Phi_0$ of rank $r$ and its Weyl group $W_0$, but it also includes the associated alcove picture. We denote by $\theta$ the highest root in $\Phi_0$, and let $\alpha_0 := -\theta$. Also, let $[m] := \{1, 2, \dotsc, m\}$ and $\mathfrak{h}_{\reals} := \reals \otimes P$.

Consider the affine hyperplanes $H_{\beta,k} := \{ \lambda \in \mathfrak{h}_{\reals} \mid \inner{\lambda}{\beta^{\vee}} = k\}$. Recall an \defn{alcove} is a connected component of $\mathfrak{h}_{\reals} \setminus \left( \bigcup_{\beta \in \Phi_0} \bigcup_{k\in\ZZ} H_{\beta,k} \right)$, and the \defn{fundamental alcove} is
\[
A_{\circ} := \{ \lambda \in \mathfrak{h}_{\reals} \mid 0 < \inner{\lambda}{\alpha_i^{\vee}} < 1 \text{ for all } i \in I_0 \}.
\]
We say that two alcoves are \defn{adjacent} if they are distinct and have a common wall. Given a pair of adjacent alcoves $A$ and $B$, we write $A \stackrel{\beta}{\longrightarrow} B$ if the
common wall is contained in the affine hyperplane $H_{\beta,k}$, for some $k \in \ZZ$, and the root
$\beta \in \Phi$ points in the direction from $A$ to $B$.

	An  \defn{alcove path} is a {sequence of alcoves} $(A_0, A_1, \dotsc, A_m)$ such that
	$A_{j-1}$ and $A_j$ are adjacent, for $j=1,\dotsc, m.$ We say that an alcove path 
	is \defn{reduced} if it has minimal length among all alcove paths from $A_0$ to $A_m$.
	Let $A_{\lambda} = A_{\circ}+\lambda$ be the translation of the fundamental alcove $A_{\circ}$ by the weight $\lambda$.
	
	The sequence of roots $(\beta_1, \beta_2, \dotsc, \beta_m)$ is called a
	\defn{$\lambda$-chain} if 
	\[	
	A_0=A_{\circ} \xrightarrow[\hspace{25pt}]{-\beta_1} A_1
	\xrightarrow[\hspace{25pt}]{-\beta_2} \cdots
	\xrightarrow[\hspace{25pt}]{-\beta_m} A_m = A_{-\lambda}
	\]
	is a reduced alcove path.

We now fix a dominant weight $\lambda$ and an alcove path $\Pi = (A_0, \dotsc, A_m)$ from 
$A_0 = A_{\circ}$ to $A_m = A_{-\lambda}$. Note that $\Pi$ is determined by the corresponding $\lambda$-chain $\Gamma := (\beta_1, \dots, \beta_m)$, which consists of positive roots. 
We let $r_i : =s_{\beta_i}$, and let $\widehat{r_i}$ be the affine reflection in the hyperplane containing the common face of $A_{i-1}$ and $A_i$, for $i = 1, \dotsc, m$; in other words, 
$\widehat{r}_i := s_{\beta_i,-l_i}$, where $l_i := \left\lvert \left\{ j < i \mid \beta_j = \beta_i \right\} \right\rvert$. 
We define $\widetilde{l}_i := \inner{\lambda}{\beta_i^{\vee}} -l_i = \left\lvert \left\{ j \geq i \mid \beta_j = \beta_i \right\} \right\rvert$.

Let $J = \left\{ j_1 < j_2 < \cdots < j_s \right\}$  be a subset of $[m]$. The elements of $J$ are called \defn{folding positions}. We fold $\Pi$ in the hyperplanes corresponding to these positions and obtain a folded path. Like $\Pi$,  the folded path can be recorded by a sequence of roots, namely 
$\Gamma(J) = \left( \gamma_1,\gamma_2, \dotsc, \gamma_m \right)$, where 
	 \begin{equation}\label{defw}
	 \gamma_k:=r_{j_1}r_{j_2} \dotsm r_{j_p}(\beta_k)\,,
	 \end{equation}
	  with $j_p$ the 
	 largest folding position less than $k$. 
	 We define $\gamma_{\infty} := r_{j_1} r_{j_2} \dotsm r_{j_s}(\rho)$.
	 Upon folding, the hyperplane separating the alcoves $A_{k-1}$ and $A_k$ in $\Pi$ is mapped to 
	\begin{equation}\label{deflev}
	H_{\lvert\gamma_k\rvert,-\l{k}} = \widehat{r}_{j_1}\widehat{r}_{j_2} \dotsm \widehat{r}_{j_p}(H_{\beta_k,-l_k})\,,
	\end{equation}
	for some $\l{k}$, which is defined by this relation.

	Given $i \in J$, we  say that $i$ is a \defn{positive folding position} if 
	 $\gamma_i>0$, and a \defn{negative folding position} if $\gamma_i<0$. 
	 We denote the positive folding positions by $J^{+}$, and the negative 
	 ones by $J^{-}$.
	 We call 
	$\wt(J) := -\widehat{r}_{j_1}\widehat{r}_{j_2} \dotsm \widehat{r}_{j_s}(-\lambda)$ the 
	\defn{weight} of $J$.
%\begin{proposition}
	%Given $J=\left\{ j_1 < j_2 < \cdots < j_s \right\} \subseteq [m]$, 
	%let $w_i:= r_{j_1}r_{j_2}\dots r_{j_i}$. We have $\gamma_{j_i} \in J^{+}$
	%%$\gamma_{j_i}=w_{i-1}(\beta_{j_i})>0$ 
	%if and only if $\ell(w_{i-1})<\ell(w_{i})$.
%\end{proposition}
%\begin{proof}
	%We have $\gamma_{j_i}=w_{i-1}(\beta_{j_i})>0$ if and only if
	%$\ell(w_{i-1})< \ell(w_i)$ [\cite{humrgc}, Proposition 5.7].
%\end{proof}

	A subset $J = \left\{ j_1 < j_2 < \cdots < j_s \right\} \subseteq [m]$ (possibly empty)
 is an \defn{admissible subset} if
we have the following path in the quantum Bruhat graph on $W_0$:
\begin{equation}
	\label{eqn:admissible}
	1 \xrightarrow[\hspace{25pt}]{\beta_{j_1}} r_{j_1} \xrightarrow[\hspace{25pt}]{\beta_{j_2}} r_{j_1}r_{j_2} 
	\xrightarrow[\hspace{25pt}]{\beta_{j_3}} \cdots \xrightarrow[\hspace{25pt}]{\beta_{j_s}} r_{j_1}r_{j_2}\cdots r_{j_s}=:\phi(J)\,.
\end{equation}
We call $\Gamma(J)$ an \defn{admissible folding} and $\phi(J)$ its \defn{final direction}. 
We let $\A(\Gamma)$ be the collection of admissible subsets.

\begin{remark}\label{spec}
Positive and negative folding positions correspond to up and down steps (in Bruhat order) in the chain~\eqref{eqn:admissible}, respectively.
\end{remark}

We now define the crystal operators on $\A( \Gamma )$. 
Given $J\subseteq [m]$ and
$\alpha\in \Phi$, we will use the following notation:
	\[ 
	 I_\alpha = I_{\alpha}(J) := \left\{ i \in [m] \: \mid \: \gamma_i = \pm \alpha \right\}\,, \qquad %L_{\alpha}=L_{\alpha}(\Delta) := \left\{ \l{i} \, | \, i \in I_{\alpha} \right\}, \\
\widehat{I}_\alpha = \widehat{I}_{\alpha}(J) := I_{\alpha} \cup \{\infty\}\,, %\quad \widehat{L}_{\alpha} = \widehat{L}_{\alpha}(\Delta) := L_{\alpha} \cup \{l_\alpha^{\infty} \},
\]
and $l_{\alpha}^{\infty} := \inner{{\rm wt}(J)}{\sgn(\alpha)\alpha^{\vee}}$.
%We will use $\widehat{L}_{\alpha}$ to define the root operators on
%admissible subsets. %in the following sections. 
The following graphical representation of the heights $l_i^J$ for $i\in{I}_\alpha$ and $l_{\alpha}^{\infty}$ %$\widehat{L}_{\alpha}$ 
is useful for defining the crystal operators.
Let 
\[
\widehat{I}_{\alpha}= \left\{ i_1 < i_2 < \dots < i_n <i_{n+1}=\infty \right\}
	\, \text{ and  } \,
	\epsilon_i := 
	\begin{cases}
		\,\,\,\, 1 &\text{ if } i \notin J\,, \\
		-1 & \text { if } i \in J\,.
	\end{cases}
\]
If $\alpha > 0$, we define the continuous piecewise linear function 
$g_{\alpha} \colon \left[ 0, n+\frac{1}{2} \right] \to \reals$ by
\begin{equation}
	\label{eqn:piecewise-linear_graph}
	g_\alpha(0)= -\frac{1}{2}, \qquad g'_{\alpha}(x)=
	\begin{cases}
		\sgn(\gamma_{i_k}) & \text{if } x \in (k-1,k-\frac{1}{2}),\, k = 1, \dotsc, n,\\
		\epsilon_{i_k}\sgn(\gamma_{i_k}) & 
		\text{if } x \in (k-\frac{1}{2},k),\, k=1,\dotsc,n, \\
		\sgn(\inner{\gamma_{\infty}}{\alpha^{\vee}}) &
		\text{if } x \in (n,n+\frac{1}{2}).
	\end{cases}
\end{equation}
If $\alpha<0$, we define $g_{\alpha}$ to be the graph obtained by reflecting $g_{-\alpha}$ in
the $x$-axis.
For any $\alpha$ we have 
\begin{equation}
	\label{eqn:graph_height}
	\sgn(\alpha)\l{i_k} = g_{\alpha}\!\left(k-\frac{1}{2}\right), \, k = 1, \dotsc, n, \ 
	\text{ and } \ 
	\sgn(\alpha)l_{\alpha}^{\infty}:=
	\inner{{\rm wt}(J)}{\alpha^{\vee}} = g_{\alpha}\!\left(n+\frac{1}{2}\right).
\end{equation}

Let $J$ be an admissible subset. 
Let $\delta_{i,j}$ be the Kronecker delta function.
Fix $p$ in $\{0, \dotsc, r\}$, so $\alpha_p$ is a simple root if $p>0$, or $-\theta$ if $p=0$.
Let $M$ be the maximum of $g_{{\alpha}_p}$, which is known to be a nonnegative integer.
Let $m := \min\{ i \in \widehat{I}_{{\alpha}_p} \mid \sgn({\alpha}_p)\l{i} = M \}$.
It turns out that, if $M \geq \delta_{p,0}$, then we have either $m\in J$ or $m=\infty$; furthermore, if $M > \delta_{p,0}$, then $m$ has a predecessor $k$ in $\widehat{I}_{{\alpha}_p}$ and $k \notin J$. We define
\begin{equation}
	\label{eqn:rootF} 
	f_p(J):= 
	\begin{cases}
		(J \setminus \left\{ m \right\}) \cup \{ k \} & \text{if $M>\delta_{p,0}$}\,, \\
				\Bzero & \text{otherwise}\,.
	\end{cases}
\end{equation}
Now we define $e_p$.
Assuming that $M > \inner{\wt(J)}{\alpha_p^{\vee}}$, let $k := \max\{i \in I_{\alpha_p} \mid \sgn(\alpha_p) \l{i} = M\}$,
and let $m$ be the successor of $k$ in $\widehat{I}_{{\alpha}_p}$. Assuming also that $M\ge\delta_{p,0}$, it turns out that we have $k\in J$ and either $m\not\in J$ or $m=\infty$. 
Define 
\begin{equation}
	\label{eqn:rootE}
	e_p(J):= 
	\begin{cases}
		(J \setminus \left\{ k \right\}) \cup \{ m \} & \text{if }
		M > \inner{\wt(J)}{\alpha_p^{\vee}} \text{ and } M \geq \delta_{p,0}  \\
				\Bzero & \text{otherwise}\,.
	\end{cases}
\end{equation}
In the above definitions, we use the
convention that $J \setminus \left\{ \infty \right\} = J \cup \left\{ \infty \right\} = J$. 

We recall one of the main results in~\cite{LNSSS14II}, cf.\ also~\cite{LNSSS14,LL18}. In the setup of untwisted affine root systems,  consider the tensor product of KR crystals $B = \bigotimes_{j=1}^{N} B^{p_j,1}$. Let $\lambda = \omega_{p_1} + \cdots + \omega_{p_N}$, and let $\Gamma$ be \emph{any} $\lambda$-chain. 

\begin{thm}[{\cite{LNSSS14II,LL18}}]\label{mainconj} 
	The (abstract) crystal $\A(\Gamma)$ is isomorphic to $\widetilde{DD}_1(B)$ via a specific weight-preserving bijection $\Psi$. 
\end{thm}

Based on the above discussion and notation, we give a modified crystal structure on $\A(\Gamma)$ such that the result is isomorphic to $\widetilde{DD}_{\ell}(B)$.
Let $\A_\ell(\Gamma)$ be the set $\A(\Gamma)$ with crystal operators defined by
\begin{subequations}
\label{eq:qalcove_crystal_ops}
\begin{align}
	f_p(J)&:= 
	\begin{cases}
		(J \setminus \left\{ m \right\}) \cup \{ k \} & \text{ if $M>\ell\delta_{p,0}$}\,, \\
				\Bzero & \text{ otherwise}\,,
	\end{cases}\\
	e_p(J)&:= 
	\begin{cases}
		(J \setminus \left\{ k \right\}) \cup \{ m \} & \text{ if }
		M>\inner{{\rm wt}(J)}{{\alpha}_p^{\vee}} \text{ and } M \geq \ell\delta_{p,0}\,,  \\
				\Bzero & \text{ otherwise}\,.
	\end{cases}
\end{align}
\end{subequations}
In particular, we have $\A_1(\Gamma) = \A(\Gamma)$. 

\begin{prop}
The map $\Psi$ from Theorem~{\rm \ref{mainconj}} restricts to a crystal isomorphism $\Psi_{\ell} \colon \A_{\ell}(\Gamma) \to \widetilde{DD}_{\ell}(B)$.
\end{prop}

\begin{proof}
It is clear that $\Psi_{\ell}$ is a bijection since, as sets, $\A_{\ell}(\Gamma) = \A(\Gamma)$ and $\widetilde{DD}_{\ell}(B) = \widetilde{DD}_1(B)$. Thus, it remains to show $\Psi_{\ell}$ commutes with the crystal operators.

Assuming $\ell\ge 2$, we have $\widetilde{DD}_{\ell}(B)=\widetilde{DD}_{l-1}(\widetilde{DD}_1(B))$, \textit{i.e.}, the crystal $\widetilde{DD}_{\ell}(B)$ is obtained from $\widetilde{DD}_1(B)$ by removing the last $\ell-1$ edges in a $0$-string. Let $\varphi_0$ be the crystal $\varphi$-function for $\widetilde{DD}_1(B)$, see Section~\ref{sec:crystals}.
By~\cite[Theorem~3.9]{LL15}, we have 
\[
\varphi_0(J)=\max(M-1,0)\,.
\]
In $\A(\Gamma)$ we have to redefine as $0$ every $f_0(J)\ne 0$ with $\varphi_0(J)\le\ell-1$, but this condition is equivalent to $M\le\ell$. Similarly, we have to redefine as $0$ every $e_0(J)\ne 0$ with $\varphi_0(J)\le\ell-2$, but this condition is equivalent to $M<\ell$. The fact that the crystal operators $f_p$ and $e_p$ in $\A_\ell(\Gamma)$ commute with $\Psi_{\ell}$ now follows from~\eqref{eqn:rootF} and~\eqref{eqn:rootE}, respectively. 
\end{proof}

\begin{remarks}
\mbox{}
\begin{itemize}
\item[(1)] Using the setup of Corollary~\ref{cor:red}, let $\lambda = \sum_{j=1}^N s_j \clfw_{r_j}$, and let $\Gamma$ be any $\lambda$-chain. By the mentioned corollary and the above discussion, we can realize $\widetilde{DD}_\ell(B)$ as the connected component of the admissible subset $J=\emptyset$ in $\A_\ell(\Gamma)$.

\item[(2)] In Remark~\ref{rem:red}, we conjectured that Corollary~\ref{cor:red} extends to the exceptional types. As the quantum alcove model applies to single-column KR crystals of any untwisted affine type, we would obtain a uniform model for all (level $\ell$ dual Demazure portions of) tensor products of perfect KR crystals with a fixed level (in the mentioned types).
\end{itemize}
\end{remarks}

The following natural question arises.

\begin{problem}
How are the non-level $\ell$ dual Demazure arrows realized in the quantum alcove model? 
\end{problem}

In the quantum alcove model for tensor products of single column KR crystals, it is expected that the extra $0$-arrows will be slightly more involved. In particular, we see that the $p$-arrows currently given by~\eqref{eqn:rootF} and \eqref{eqn:rootE} change only one element in the admissible subsets. It was observed that the extra $0$-arrows change more than one entry, but there is no precise conjecture currently. On the other hand, for single columns all the 0-arrows are described in the closely related quantum LS path model, which has been bijected to the quantum alcove model~\cite{LNSSS14II}. So it would be interesting to see which of the two models would be better suited for describing the non-level $\ell$ (dual) Demazure arrows.

%=====================================================================

\section{Conjectures for $Q$-systems}
\label{sec:Q_system}

Recall the $Q$-system relations (we refer the reader to~\cite{KNS11} and references therein):
\begin{subequations}
\label{eq:Q_system}
\begin{align}
\label{eq:untwisted_Q_system}
\left(Q^{(a)}_m\right)^2 & = Q^{(a)}_{m+1} Q^{(a)}_{m-1} + \prod_{b \sim a} \prod_{k=0}^{-A_{ab} - 1} Q^{(b)}_{\left\lfloor \frac{m A_{ba} - k}{A_{ab}} \right\rfloor}\,,
\\ \label{eq:twisted_Q_system}
\left( Q^{(a)}_{m} \right)^2 & = Q^{(a)}_{m+1} Q^{(a)}_{m-1} + \prod_{b \sim a} (Q^{(b)}_{m})^{-A_{ba}}\,,
\end{align}
\end{subequations}
where~\eqref{eq:untwisted_Q_system} is for the untwisted $Q$-system and~\eqref{eq:twisted_Q_system} is for the twisted $Q$-system.

\begin{conj}
\label{conj:Q_system_Demazure}
Fix some $a \in I_0$, and let $c = \min \{ c_b \mid A_{ba} \neq 0 \}$.
Let $\ell \geq \lceil m / c \rceil$, then we have
\begin{subequations}
\begin{align}
\label{eq:untwisted_Q_Demazure}
\widetilde{D}_{\ell}\left( (B^{a,m-1})^{\otimes 2} \right) & \iso \widetilde{D}_{\ell}(B^{a,m} \otimes B^{a,m-2}) \oplus \widetilde{D}_{\ell}\left( \bigotimes_{b \sim a} \bigotimes_{k=0}^{-A_{ab} - 1} B^{b, L(k)-1} \right)\,,
\\ \label{eq:twisted_Q_Demazure}
\widetilde{D}_{\ell}\left( (B^{a,m-1})^{\otimes 2} \right) & \iso \widetilde{D}_{\ell}(B^{a,m} \otimes B^{a,m-2}) \oplus \widetilde{D}_{\ell}\left( \bigotimes_{b \sim a} (B^{b,m-1})^{\otimes -A_{ba}} \right)\,,
\end{align}
\end{subequations}
where \eqref{eq:untwisted_Q_Demazure} is for the untwisted types, \eqref{eq:twisted_Q_Demazure} is for the twisted types, and $L(k) = \left\lfloor \frac{m A_{ba} - k}{A_{ab}} \right\rfloor$.
\end{conj}

Note that it is sufficient to prove the case when $\ell \geq \lceil m / c \rceil$. 

Conjecture~\ref{conj:Q_system_Demazure} is a crystal theoretic interpretation of~\eqref{eq:Q_system} (with renormalized indices). We note that if we branch to $U_q(\g_0)$-crystals (or if we took $\ell \gg 1$), then Conjecture~\ref{conj:Q_system_Demazure} becomes precisely the statement that classical characters of KR crystals satisfy the $Q$-system~\cite{Hernandez10}. Thus, Conjecture~\ref{conj:Q_system_Demazure} is a strengthening of~\cite{Hernandez10}.

Theorem~\ref{thm:demazure_decomposition} says that the crystals in Conjecture~\ref{conj:Q_system_Demazure} are isomorphic to a disjoint union of Demazure crystals. Theorem~\ref{thm:red} implies that the components containing the maximal/minimal elements are isomorphic. However, one would need to precisely enumerate \emph{all} connected components and check their maximal and minimal weights in order to show Conjecture~\ref{conj:Q_system_Demazure}. We have verified Conjecture~\ref{conj:Q_system_Demazure} on a number of examples in different types by using \textsc{SageMath}~\cite{sage}.

\bibliographystyle{alpha}

\begin{thebibliography}{KKM{\etalchar{+}}92b}

\bibitem[Bou02]{Bourbaki02}
N. Bourbaki.
\newblock {\em Lie groups and {L}ie algebras. {C}hapters 4--6}.
\newblock Elements of Mathematics (Berlin). Springer-Verlag, Berlin, 2002.
\newblock Translated from the 1968 French original by Andrew Pressley.

\bibitem[BS17]{BS17}
D. Bump and A. Schilling.
\newblock {\em Crystal bases}.
\newblock World Scientific Publishing Co. Pte. Ltd., Hackensack, NJ, 2017.
\newblock Representations and combinatorics.

\bibitem[BS19]{BS19II}
R. Biswal and T. Scrimshaw
\newblock Existence of {K}irillov-{R}eshetikhin crystals for multiplicity free nodes.
\newblock Preprint, \arxiv{1902.00769}, 2019.

\bibitem[Car05]{Carter05}
R.~W. Carter.
\newblock {\em Lie algebras of finite and affine type}, volume~96 of {\em
  Cambridge Studies in Advanced Mathematics}.
\newblock Cambridge University Press, Cambridge, 2005.

\bibitem[CP95]{CP95}
V. Chari and A. Pressley.
\newblock Quantum affine algebras and their representations.
\newblock In {\em Representations of groups ({B}anff, {AB}, 1994)}, volume~16
  of {\em CMS Conf. Proc.}, pages 59--78. Amer. Math. Soc., Providence, RI,
  1995.

\bibitem[CP98]{CP98}
V. Chari and A. Pressley.
\newblock Twisted quantum affine algebras.
\newblock {\em Comm. Math. Phys.}, 196:461--476, 1998.

\bibitem[FL06]{FL06}
G.~Fourier and P.~Littelmann.
\newblock Tensor product structure of affine {D}emazure modules and limit
  constructions.
\newblock {\em Nagoya Math. J.}, 182:171--198, 2006.

\bibitem[FL07]{FL07}
G.~Fourier and P.~Littelmann.
\newblock Weyl modules, {D}emazure modules, {KR}-modules, crystals, fusion
  products and limit constructions.
\newblock {\em Adv. Math.}, 211:566--593, 2007.

\bibitem[FOS09]{FOS09}
G. Fourier, M. Okado, and A. Schilling.
\newblock Kirillov-{R}eshetikhin crystals for nonexceptional types.
\newblock {\em Adv. Math.}, 222:1080--1116, 2009.

\bibitem[FOS10]{FOS10}
G. Fourier, M. Okado, and A. Schilling.
\newblock Perfectness of {K}irillov-{R}eshetikhin crystals for nonexceptional
  types.
\newblock {\em Contemp. Math.}, 506:127--143, 2010.

\bibitem[FSS07]{FSS07}
G. Fourier, A. Schilling, and M. Shimozono.
\newblock Demazure structure inside {K}irillov-{R}eshetikhin crystals.
\newblock {\em J. Algebra}, 309:386--404, 2007.

\bibitem[FW04]{FW04}
W. Fulton and C. Woodward.
\newblock On the quantum product of Schubert classes.
\newblock {\em J. Algebraic Geom.}, 13:641--661, 2004.

\bibitem[Her10]{Hernandez10}
D. Hernandez.
\newblock Kirillov-{R}eshetikhin conjecture: the general case.
\newblock {\em Int. Math. Res. Not.}, (1):149--193, 2010.

\bibitem[HKO{\etalchar{+}}99]{HKOTY99}
G. Hatayama, A.~Kuniba, M.~Okado, T.~Takagi, and Y.~Yamada.
\newblock Remarks on fermionic formula.
\newblock In {\em Recent developments in quantum affine algebras and related
  topics ({R}aleigh, {NC}, 1998)}, volume 248 of {\em Contemp. Math.}, pages
  243--291. Amer. Math. Soc., Providence, RI, 1999.

\bibitem[HKO{\etalchar{+}}02]{HKOTT02}
G. Hatayama, A. Kuniba, M. Okado, T. Takagi, and Z. Tsuboi.
\newblock Paths, crystals and fermionic formulae.
\newblock In {\em Math{P}hys Odyssey, 2001}, volume~23 of {\em Prog. Math.
  Phys.}, pages 205--272. Birkh\"auser Boston, Boston, MA, 2002.

\bibitem[HN06]{HN06}
D. Hernandez and H. Nakajima.
\newblock Level 0 monomial crystals.
\newblock {\em Nagoya Math. J.}, 184:85--153, 2006.

\bibitem[Jos03]{Joseph03}
A. Joseph.
\newblock A decomposition theorem for {D}emazure crystals.
\newblock {\em J. Algebra}, 265:562--578, 2003.

\bibitem[JS10]{JS10}
B. Jones and A. Schilling.
\newblock Affine structures and a tableau model for {$E_6$} crystals.
\newblock {\em J. Algebra}, 324:2512--2542, 2010.

\bibitem[Kac90]{kac90}
V.~G. Kac.
\newblock {\em Infinite-dimensional {L}ie algebras}.
\newblock Cambridge University Press, Cambridge, third edition, 1990.

\bibitem[Kas90]{K90}
M. Kashiwara.
\newblock Crystalizing the {$q$}-analogue of universal enveloping algebras.
\newblock {\em Comm. Math. Phys.}, 133:249--260, 1990.

\bibitem[Kas91]{K91}
M. Kashiwara.
\newblock On crystal bases of the $q$-analogue of universal enveloping
  algebras.
\newblock {\em Duke Math. J.}, 63:465--516, 1991.

\bibitem[Kas94]{K94}
M.~Kashiwara.
\newblock Crystal bases of modified quantized enveloping algebra.
\newblock {\em Duke Math. J.}, 73:383--413, 1994.

\bibitem[Kas96]{K96}
M. Kashiwara.
\newblock Similarity of crystal bases.
\newblock In {\em Lie algebras and their representations ({S}eoul, 1995)},
  volume 194 of {\em Contemp. Math.}, pages 177--186. Amer. Math. Soc.,
  Providence, RI, 1996.

\bibitem[Kas02]{K02}
M. Kashiwara.
\newblock On level-zero representations of quantized affine algebras
\newblock {\em Duke Math. J.}, 112(1):117--175, 2002.

\bibitem[Kas05]{K05}
M. Kashiwara.
\newblock Level zero fundamental representations over quantized affine algebras
  and {D}emazure modules.
\newblock {\em Publ. Res. Inst. Math. Sci.}, 41:223--250, 2005.

\bibitem[KKM{\etalchar{+}}92a]{KKMMNN91}
S.-J. Kang, M. Kashiwara, K.~C. Misra, T. Miwa, T.
  Nakashima, and A. Nakayashiki.
\newblock Affine crystals and vertex models.
\newblock In {\em Infinite analysis, {P}art {A}, {B} ({K}yoto, 1991)},
  volume~16 of {\em Adv. Ser. Math. Phys.}, pages 449--484. World Sci. Publ.,
  River Edge, NJ, 1992.

\bibitem[KKM{\etalchar{+}}92b]{KKMMNN92}
S.-J. Kang, M. Kashiwara, K.~C. Misra, T. Miwa, T.
  Nakashima, and A. Nakayashiki.
\newblock Perfect crystals of quantum affine {L}ie algebras.
\newblock {\em Duke Math. J.}, 68:499--607, 1992.

\bibitem[Kle98]{Kleber98}
M. Kleber.
\newblock {\em Finite dimensional representations of quantum affine algebras}.
\newblock ProQuest LLC, Ann Arbor, MI, 1998.
\newblock Thesis (Ph.D.)--University of California, Berkeley.

\bibitem[KMOY07]{KMOY07}
M.~Kashiwara, K.~C. Misra, M.~Okado, and D.~Yamada.
\newblock Perfect crystals for {$U_q(D^{(3)}_4)$}.
\newblock {\em J. Algebra}, 317:392--423, 2007.

\bibitem[KNS11]{KNS11}
A. Kuniba, T. Nakanishi, and J. Suzuki.
\newblock {$T$}-systems and {$Y$}-systems in integrable systems.
\newblock {\em J. Phys. A}, 44:103001, 146, 2011.

\bibitem[LL15]{LL15}
C. Lenart and A. Lubovsky.
\newblock A generalization of the alcove model and its applications.
\newblock {\em J. Algebraic Combin.}, 41:751--783, 2015.

\bibitem[LL18]{LL18}
\newblock  C.~Lenart and A. Lubovsky.
\newblock A uniform realization of the combinatorial $R$-matrix.
\newblock {\em Adv. Math.}, 334:151--183, 2018.

\bibitem[LLM02]{LLM02}
V. Lakshmibai, P. Littelmann, and P. Magyar.
\newblock Standard monomial theory for {B}ott-{S}amelson varieties.
\newblock {\em Compos. Math.}, 130:293--318, 2002.

\bibitem[LNS{\etalchar{+}}15]{LNSSS14}
C. Lenart, S. Naito, D. Sagaki, A. Schilling, and M. Shimozono.
\newblock A uniform model for {K}irillov-{R}eshetikhin crystals {I}: {L}ifting
  the parabolic quantum {B}ruhat graph.
\newblock {\em Int. Math. Res. Not.}, (7):1848--1901, 2015.

\bibitem[LNS{\etalchar{+}}17]{LNSSS14II}
C. Lenart, S. Naito, D. Sagaki, A. Schilling, and M. Shimozono.
\newblock A uniform model for {K}irillov-{R}eshetikhin crystals {II}. {A}lcove
  model, path model, and {$P=X$}.
\newblock {\em Int. Math. Res. Not.}, (14):4259--4319, 2017.

\bibitem[LNS{\etalchar{+}}17]{LNSSS15}
C. Lenart, S. Naito, D. Sagaki, A. Schilling, and M. Shimozono.
\newblock A uniform model for {K}irillov-{R}eshetikhin crystals {III}:
  {N}onsymmetric {M}acdonald polynomials at {$t=0$} and {D}emazure characters.
\newblock {\em Transform. Groups}, 22:1041--1079, 2017.

\bibitem[LS19]{LS17}
X. Liu and T. Scrimshaw.
\newblock A uniform approach to soliton cellular automata using rigged
  configurations.
\newblock {\em Ann. Henri Poincar\'e}, 20(4):1175--1215, 2019.

\bibitem[Nao12]{Naoi12}
K. Naoi.
\newblock Weyl modules, {D}emazure modules and finite crystals for non-simply laced type.
\newblock {\em Adv. Math.}, 231:1546--1571, 2012.

\bibitem[Nao13]{Naoi13}
K. Naoi.
\newblock Demazure crystals and tensor products of perfect
  {K}irillov-{R}eshetikhin crystals with various levels.
\newblock {\em J. Algebra}, 374:1--26, 2013.

\bibitem[Nao17]{Naoi17}
K. Naoi.
\newblock Existence of {K}irillov-{R}eshetikhin crystals of type {$G_2^{(1)}$}
  and {$D_4^{(3)}$}.
\newblock {\em J. Algebra} 512:47--65, 2018.

\bibitem[Nao19]{Naoi19}
K. Naoi.
\newblock Existence of Kirillov-Reshetikhin crystals of type {$E_6^{(1)}$}.
\newblock Preprint, \arxiv{1903.11681}, 2019.

\bibitem[NS08a]{NS08II}
S. Naito and D. Sagaki.
\newblock Crystal structure on the set of {L}akshmibai-{S}eshadri paths of an
  arbitrary level-zero shape.
\newblock {\em Proc. Lond. Math. Soc.}, 96:582--622, 2008.

\bibitem[NS08b]{NS08}
S. Naito and D. Sagaki.
\newblock Lakshmibai-{S}eshadri paths of level-zero shape and one-dimensional
  sums associated to level-zero fundamental representations.
\newblock {\em Compos. Math.}, 144:1525--1556, 2008.

\bibitem[Nom16]{Nom16}
F. Nomoto. 
\newblock Quantum Lakshmibai-Seshadri paths and the specialization of Macdonald 
polynomials at $t=0$ in type $A_{2n}^{(2)}$.
\newblock Preprint, \arxiv{1606.01067}, 2016.

\bibitem[Oka13]{Okado13}
M. Okado.
\newblock Simplicity and similarity of {K}irillov-{R}eshetikhin crystals.
\newblock In {\em Recent developments in algebraic and combinatorial aspects of
  representation theory}, volume 602 of {\em Contemp. Math.}, pages 183--194.
  Amer. Math. Soc., Providence, RI, 2013.

\bibitem[OS08]{OS08}
M. Okado and A. Schilling.
\newblock Existence of {K}irillov-{R}eshetikhin crystals for nonexceptional
  types.
\newblock {\em Represent. Theory}, 12:186--207, 2008.

\bibitem[OSS03]{OSS03II}
M. Okado, A. Schilling, and M. Shimozono.
\newblock Virtual crystals and {K}leber's algorithm.
\newblock {\em Comm. Math. Phys.}, 238:187--209, 2003.

\bibitem[OSS13]{OSS13}
M. Okado, R. Sakamoto, and A. Schilling.
\newblock Affine crystal structure on rigged configurations of type {$D_n^{(1)}$}
\newblock {\em J. Algebraic Combin.}, 37:571--599, 2013.

\bibitem[OSS18]{OSS17}
M. Okado, A. Schilling, and T. Scrimshaw.
\newblock Rigged configuration bijection and proof of the {$X = M$} conjecture
  for nonexceptional affine types.
\newblock {\em J. Algebra}, 516:1--37, 2018.

\bibitem[PS18]{PS15}
J. Pan and T. Scrimshaw.
\newblock Virtualization map for the {L}ittelmann path model.
\newblock {\em Transform. Groups}, 23(4):1045--1061, 2018.

\bibitem[Sage17]{sage}
The~Sage Developers.
\newblock {\em {S}age {M}athematics {S}oftware ({V}ersion 8.3)}.
\newblock The Sage Development Team, 2018.
\newblock \url{http://www.sagemath.org}.

\bibitem[SCc08]{combinat}
The {S}age-{C}ombinat community.
\newblock {S}age-{C}ombinat: enhancing {S}age as a toolbox for computer
  exploration in algebraic combinatorics, 2008.
\newblock \url{http://combinat.sagemath.org}.

\bibitem[Sch06]{S06}
A. Schilling.
\newblock Crystal structure on rigged configurations.
\newblock {\em Int. Math. Res. Not.}, pages Art. ID 97376, 27, 2006.

\bibitem[SS15]{SchillingS15}
A. Schilling and T. Scrimshaw.
\newblock Crystal structure on rigged configurations and the filling map for
  non-exceptional affine types.
\newblock {\em Electron. J. Combin.}, 22:Research Paper 73, 56, 2015.

\bibitem[ST12]{ST12}
A. Schilling and P. Tingley.
\newblock Demazure crystals, {K}irillov-{R}eshetikhin crystals, and the energy
  function.
\newblock {\em Electron. J. Combin.}, 19:Paper 4, 42, 2012.

\bibitem[SW10]{SW10}
A. Schilling and Q. Wang.
\newblock Promotion operator on rigged configurations of type {$A$}.
\newblock {\em Electron. J. Combin.}, 17:Research Paper 24, 43, 2010.

\bibitem[Yam98]{Yamane98}
S. Yamane.
\newblock Perfect crystals of {$U_q(G^{(1)}_2)$}.
\newblock {\em J. Algebra}, 210:440--486, 1998.

\end{thebibliography}
\newcommand{\etalchar}[1]{$^{#1}$}

\end{document}